\documentclass[a4paper,centertags,oneside,12pt]{amsart}
\usepackage{mathrsfs}
\usepackage{appendix}
\usepackage{amssymb}
\usepackage{fancyhdr}
\usepackage{charter}
\usepackage{typearea}
\usepackage{pdfsync}
\usepackage{mathrsfs}
\usepackage[a4paper,top=3cm,bottom=3cm,left=2cm,right=2cm]{geometry}
\usepackage{enumitem}
\usepackage{color}

\usepackage[pdftex, colorlinks=true, urlcolor=black]{hyperref}
\hypersetup{  citecolor=black}

\setlist[1]{itemsep=5pt}
\newcommand{\comment}[1]{}
\makeatletter

\def\@setcopyright{}

\def\serieslogo@{}
      \makeatother

\newcommand*\Laplace{\mathop{}\!\mathbin\bigtriangleup}
\newcommand*\DLaplace{\mathop{}\!\mathbin\Box}
\newcommand*\str{{\mathrm{STr}\,}}
\newcommand*\trace{{\mathrm{Tr\,}}}
\newcommand*\ind{{\mathrm{ind\,}}}
\newcommand*\ch{{\mathrm{ch\,}}}
\newcommand*\dimension{{\mathrm{dim}}}

\newcommand*\expo{{\mathrm{exp\,}}}
\newcommand*\sone{{\mathbb{S}^1}}
\newcommand{\ortho}[1]{{\mathscr{SO}(#1)}}

\newcommand{\coloneqq}{:=}

\theoremstyle{plain}
\newtheorem{thm}{Theorem}[subsection]

\newtheorem{lemma}[thm]{Lemma}

\newtheorem{prop}[thm]{Proposition}

\theoremstyle{definition}
\newtheorem{defn}[thm]{Definition}
\newtheorem{nota}[thm]{Notation}
\theoremstyle{remark}
\newtheorem{rmk}[thm]{Remark}

\numberwithin{equation}{subsection}

\renewcommand{\epsilon}{\varepsilon}

\begin{document}
\title[A Local Index Theorem of Transversal Type on Manifolds with Locally free $\sone$-action]{A Local Index Theorem of Transversal Type on Manifolds with Locally free $\sone$-action}

\author{Dung-Cheng Lin}
\address{Department of Mathematics, National Taiwan University, Taipei 10617,
Taiwan}
\email{dungchenglin@ntu.edu.tw}

\author{I-Hsun Tsai}
\address{Department of Mathematics, National Taiwan University, Taipei 10617,
Taiwan}
\email{ihtsai@math.ntu.edu.tw}

\maketitle

\begin{abstract}
We study an index of a transversal Dirac operator on an odd-dimensional manifold $X$ with locally free $\sone$-action. One difficulty of using heat kernel method lies in the understanding of the asymptotic expansion as $t\to 0^+$. By a probabilistic approach via the Feynman-Kac formula, the transversal heat kernel on $X$ can be linked to the ordinary heat kernel for functions on the orbifold $M=X/\sone$ which is more tractable. After some technical results for a uniform bound estimate as $t\to 0^+$, we are reduced from the transversal, orbifold situation to the classical situation particularly at points of the principal stratum. One application asserts that for a certain class of spin orbifolds $M$, to the classical index problem of Kawasaki in the Riemannian setting the net contributions arising from the lower-dimensional strata beyond the principal one vanish identically.

\end{abstract}

{}

\section{Introduction} \label{sec:intro}
In \cite{cheng2015heat} a local index theorem is proved for a compact, oriented Cauchy-Riemann manifold $X$ with a transversal, locally free CR $\sone$-action.  If $X$ is the circle bundle of a holomorphic line bundle $L$ on a compact complex manifold $M$, this index theorem gives rise to the Hirzebruch-Riemann-Roch theorem for the dual bundle $L^\ast$ (and $(L^\ast)^m$) on $M$. A similar result also holds true in the more general case where the pair $(M,L)$ above is understood in the context of orbifold geometry. It improves a classical index formula of Kawasaki \cite{kawasaki1978the,kawasaki1979the} on certain complex orbifolds in the way as similarly indicated in the {Abstract}. Let us refer to \cite[p. 6 and p. 23]{cheng2015heat} for explanation within the CR context. Recent works in the related area under different formalisms and approaches include \cite{bruning2010equivariant}, \cite{bruning2010eta}, \cite{paradan2009index} and \cite{richardson1998asymptotics}.

In the present paper we study, in a nutshell, a transversal index theorem in the Riemannian setting with probabilistic heat kernel method. Here and henceforth $X$ denotes an odd-dimensional manifold with locally free $\sone$-action satisfying appropriate topological conditions (which are not specified here) such that a transversal Dirac operator can be defined. We would like to compute the associated local index density in this context. The main result is:\\

\noindent\textbf{Main Theorem.} \textit{($=$Theorem \ref{thm:main_result}) Given our transversal Dirac operator $D_{X,0}$ on $X$ with locally free $\,\sone$-action, the equivariant index $\ind_\sone(D_{X,0})\coloneqq \dimension\ker(D_{X,0}^+)-\dimension\ker(D_{X,0}^-)$ is given by
\[
\ind_{\sone}(D_{X,0})=\frac{p}{2\pi}\int_X\widehat A(\mathcal H)\wedge\ch\xi\wedge\omega_0,\]
where $p=\min_{u\in X}|H_u|$ denotes the least order of the isotropy subgroups $H_u$ of $\,
\sone$, $\mathcal H$ the horizontal subbundle of $\,TX$ and $\omega_0$ the (normalized) global real $1$-form along $\,\sone$-orbits. Here $\xi$ is any $\sone$-equivariant complex vector bundle (endowed with an $\sone$-invariant connection).}\\

At the end of this {Introduction}, we will indicate applications of the above result.

Our approach consists in a combination of the probabilistic method and the heat kernel one. On the heat kernel side, the authors in \cite{cheng2015heat} have treated this in depth, partly within the realm of CR and complex geometry (such as BRT coordinates) and partly within that of metric geometry (such as off-diagonal estimates from Getzler's rescaling technique as well as an insertion of ``distance'' function discovered by them). Eventually they are able to obtain an asymptotic expansion for the relevant transversal heat kernel, and then compute the local, transversal index density out of the ordinary (non-transversal) one.

Here we do not, however, aim at an extension of \cite{cheng2015heat} to the Riemannian setting, although this could be of interest in its own right. Our approach lies in a minimal use of information from the asymptotic expansion. To explain this we turn now to the probabilistic side.

The probabilistic approach to index theorems has been pioneered by J.-M. Bismut in seminal works \cite{bismut1984atiyah,bismut1984atiyah2}.   See also \cite{watanabe1990short}, \cite{watanabe1987analysis}.  
One of the key ideas is the Feynman-Kac representation of heat kernels. As far as our transversal index theorem is concerned, the formulation and method of E. P. Hsu as explained in the monograph \cite{hsu2002stochastic} are especially suitable for our purpose. But since $M=X/\sone$ may be singular and the usual extrinsic method via the Whitney embedding theorem is thus not adequate to our needs, we resort to intrinsic methods (cf. \cite{stroock2000introduction}) and adapt them to the smooth $X$. We found it possible to have a conversion of the transversal heat kernel $p_X^D(t,u,v)$ on $X$ (cf. \eqref{eq:heat_eq_for_p_X^D}) to the ordinary heat kernel $p_M(t,x,y)$ for functions on the orbifold $M=X/\sone$. It turns out that the minimal amount of information from the asymptotic expansion to be used for us is those from $p_M$ only, provided that certain technical results on orbifold geometry can be verified.

To reduce ourselves from the transversal situation to the ordinary one is almost immediate given that the $\sone$-action is globally free (so $M$ is an ordinary manifold). However, for the $\sone$-action which is only locally free as assumed here, the preceding reduction procedure cannot be warranted without a careful justification. We observe, based upon Hsu's method, that this justification can be made provided that a certain uniform bound estimate is satisfied (see \eqref{eq:boundedness}). This leads us to the ``interchangeability'' as coined in Section \ref{sec: est_HBM}. We manage to prove such analytic requirement as said, and this essentially enables us to complete the needed reduction step.

We remark that the strict heat kernel approach in \cite{cheng2015heat} for CR manifolds, yields for the asymptotic expansion of the desired heat kernel an extra ``correction term'' (not existent if the $\sone$-action is globally free), which is schematically of the form $t^{-1/2\,(\text{dim}X-1)}$ $e^{-\epsilon \widehat d(x,X_{\mathrm{sing}})^2/t}$ with $\widehat d(x,X_{\mathrm{sing}})$ a kind of distance between $x$ and higher strata $X_{\mathrm{sing}}$, cf. \cite[(1.19)]{cheng2015heat}. Since this correction (in $x$) goes unbounded as $t\to 0$ (for those $x$ as near $X_{\mathrm{sing}}$ as possible), the naive implementation cannot lead to the local index density $I(x)$. This is where the so-called off-diagonal estimate comes in at the time when taking the  supertrace of the heat kernel \cite[(5.27)]{cheng2015heat}, so that (after the 
supertrace) certain cancellation occurs among those unbounded corrections and the existence of $I(x)$ is restored. Our approach in the present paper is thus significantly distinct, although the uniform bound estimate as mentioned above should amount to effectuating the preceding cancellation issue in a different way.

For further questions such as the locally free action by other compact Lie groups $H$ or a study into asymptotic expansions and trace integrals in the Riemannian setting similar to \cite{cheng2015heat} in the CR setting, some more work needs to be done and generalized; see also \cite{park1996basic}, \cite{richardson2010transversal}. 
However the $H$-invariant part correspoinding to $m=0$ (see below for more) can be done 
in a way parallel to the present treatment without essential difficulties.    We omit the precise formulation here.  
It is known that all orbifolds can arise in this way with $H=O(n)$ (cf. \cite[p. 76]{kawasaki1978the} 
or \cite[p. 174]{duistermaat2013heat}).  
As a consequence, in some of Kawasaki's index problems on orbifolds (not necessarily  $\sone$-quotients) 
\cite{kawasaki1978the,kawasaki1979the} the net contributions 
from the lower-dimensional strata beyond the principal stratum necessarily vanish
(cf. \cite{paradan2009index} for related results).    Note that the contribution from the individual lower-dimensional
stratum may be nonvanishing (cf.   \cite[p. 185]{duistermaat2013heat}).  

In the formulation of the local index theorem given in the Main Theorem above (and Theorem \ref{thm:main_result}), we focus on the $\sone$-invariant part $\Omega_0(G)$ which belongs to the full set of Fourier components $\{\Omega_m(G)\}_{m\in\mathbb{Z}}$ (see Definition \ref{def:Fourier_component_G}). Nevertheless it is possible to extend the index theorem to all $m\in\mathbb{Z}$ in a completely analogous way. See the comments in the last paragraph of this paper, and \cite{cheng2015heat} for a similar treatment in the context of CR manifolds.

Finally let us remark that besides the application to index problems of Kawasaki as aforementioned, another advantage of the present index theorem is that it suggests an extension of the index theorem in the CR case \cite{cheng2015heat} to the ``almost CR'' case. One may think of this extension as a transversal counterpart of the classical extension of the index theorem from the case of a holomorphic vector bundle on a complex manifold to the case of a complex vector bundle on an almost complex manifold (see e.g. \cite[Theorem 3.42, p. 86]{salamon1999spin}). Such an extension in the transversal sense does not appear obvious from the methodology of \cite{cheng2015heat} due to the heavy use of the CR and complex structures involved there. Since the needed techniques for the almost complex setting provided the Riemannian setting
are fairly standard (\cite[p. 86]{salamon1999spin}), which are basically applicable to the situation here (via the Main Theorem above), we leave the details to the interested reader.

\bigskip

\textbf{Acknowledgments.} This work owes a great deal to the monograph \cite{hsu2002stochastic}; to the author Professor Elton P. Hsu we are very much grateful. Part of this joint work was done while the first named author was a research assistant in the National Center of Theoretical Sciences (NCTS) in Taiwan, which he sincerely thanks.  The second named author 
would like to thank Jih-Hsin Cheng and Chin-Yu Hsiao for related collaborations.  He was partially supported by National
Taiwan University grant no. 106-2821-C-002-001-ES.

\section{Preliminaries}

Of particular interest to us are certain index theorems for manifolds with \textit{locally free} $\sone$-action. We briefly review some basics, yet leave some of them concerning orbifolds and probability (for geometrically minded reader) 
in appendices.  

\subsection{Notations and Set-up}
\label{subsection:Notations and Set Up}
We first consider $X$ to be a principal $\sone$-bundle, so that the $\sone$-action is \textit{globally} free. The corresponding index theorem on $X$ will be identified as the Atiyah-Singer index theorem on the manifold $M=X/\sone$. This is not new, but we aim to set up the notation.

Our notation mainly follows Hsu's work \cite{hsu2002stochastic}. Let $X=(X,G,\pi)$ be a principal $\sone$-bundle of dimension $n+1$ over the base manifold $M$ with the projection $\pi:X\to M=X/\sone$. We assume that $M$ is a compact and oriented Riemannian manifold of even dimension $n=2\ell$. Denote by $\mathscr{SO}(M)$ the principal ${\rm{SO}}(n)$-bundle of oriented orthonormal frames on $M$ with the projection $\widehat\pi:\mathscr{SO}(M)\to M$. There is an $\sone$-invariant metric on $X$. We assume that this $\sone$-invariant metric induces the Riemannian metric on $M$.

The Laplace-Beltrami operator $\Laplace_M$ on functions $f\in C^\infty(M)$ is defined as usual, and the Bochner's horizontal Laplacian on $\ortho{M}$ is defined by $\Laplace^H_{\ortho{M}}=\sum_{i=1}^n H_i^2$ where $\{ H_i\}$ are the fundamental horizontal vector fields on $\ortho{M}$, cf. \cite[Section 3.1]{hsu2002stochastic}.

Let $f\in C^\infty(M)$ and $\widehat f=f\circ \pi$ its lift to $\ortho{M}$. We have the identity that for any $u\in\mathscr O(M)$, $\Laplace_Mf(\pi u)=\Laplace_{\ortho{M}}^H\widehat f(u)$ (\cite[Proposition 3.1.2]{hsu2002stochastic}). The operators $\Laplace_M$ and $\Laplace^H_{\ortho{M}}$ can be extended to act on tensor fields $\theta$ (as the covariant Laplacian) with similar relation (cf. \cite[Section 7.1]{hsu2002stochastic}).

\begin{rmk}
\label{rmk:Hodge-de Rham_Laplace}
The Hodge-de Rham Laplacian is defined by $\DLaplace_M=-(d\delta +\delta d)$, where $d$ is the exterior differentiation with its formal adjoint $\delta$.
The lifted operator $\DLaplace^H_{\mathscr{O}(M)}$ is likewise related to $\DLaplace_M$ by $\DLaplace^H_{\mathscr{O}(M)}\widehat{\theta}(\widehat x)=\widehat x^{-1}\DLaplace_M\theta (x),\; \widehat \pi \widehat x=x$. Here $\widehat x:\mathbb{R}^n\to T_xM$ is canonically extended to an isometry $\widehat x:T^{r,s}\mathbb{R}^n\to T^{r,s}_x M$ for $(r,s)$-tensors, in particular, $p$-forms. See \eqref{eq:Lplace^H_X} of Remark \ref{rmk:Laplace_M_OM} in the context of more pertinence to our interest.
\end{rmk}

Assume that $M$ is spin so we fix a spin structure on $M$ and write $\mathscr{SP}(M)$ for the associated principal bundle. Due to the spin and half-spin representations $\Laplace=\Laplace^+\oplus \Laplace^-$, one has the spin bundles $\mathscr{S}(M)=\mathscr{SP}(M)\times_{\mathrm{spin}(n)}\Laplace$ (resp. $\mathscr S(M)^\pm =\mathscr{SP}(M)\times_{{\rm{Spin}}(n)}\Laplace^\pm$)$=\mathscr S(M)^+\oplus \mathscr S(M)^-$.

Let $\xi$ be a complex vector bundle on $M$ equipped with a connection $\nabla^\xi$. Denote by $\nabla^M$ the Riemannian connection on $M$. The twisted bundle $G=\mathscr S(M)\otimes \xi$ is endowed with the product connection, simply denoted as $\nabla$. We have $G=G^+\oplus G^-,\quad G^\pm=\mathscr S(M)^\pm\otimes \xi$. The Dirac operator $D$ on $G$ is defined as a series of compositions:
\begin{equation}
\label{eq:Dirac_composition}
D:\Omega(G)\overset{\nabla}{\longrightarrow}\Omega(T^\ast M\otimes G)\overset{\text{dual}}{\longrightarrow}\Omega(TM\otimes G)\overset{c}{\longrightarrow}\Omega(G),
\end{equation}
where $\Omega(G)$ denotes the space of $G$-valued differential forms on $M$ and $c$ the Clifford multiplication.

\subsection{An Index Theorem on a Principal $\mathbf{\sone}$-bundle ${X}$}
\label{subsection:Atiyah-Singer Index Theorem}
We shall now describe a relation between our index theorem on $X$ and the classical Atiyah-Singer index theorem on $M=X/\sone$. See \cite{cheng2015heat} for a similar result in the context of CR manifolds.

Let $D$ on $M$ be as above and $D_X$ be the first order, $\sone$-equivariant differential operator or the so-called transversal Dirac operator such that $-D_X^2$ is transversally elliptic; see \cite{prokhorenkov2011natural} and Subsection \ref{subsection:transversal_D_X}. Let $p_X^D(t,u,v)$ be the transversal heat kernel for $-D_X^2$ on $X$ (see Proposition \ref{prop:heat_eq_D_X}).

Our idea is firstly to reformulate the McKean-Singer formula for an $\sone$-equivariant index of $D_X$. That is, we will be engaged in the following (cf. \eqref{eq:McKean-Singer formula}):
\begin{equation}
\label{eq:global_index_X}
\ind_\sone (D_X)=\int_X\str p_X^D(t,u,u)\mathrm du,\quad \forall t>0,
\end{equation}
and would like to write the above integral as one down on $M$.
For the principal $\sone$-bundle $X$ it is true that, after a suitable normalization on the metric of $X$ (see Section \ref{sec:const}),
\begin{equation}
\label{eq: def of p_X_global}
p_X^D(t,u,u)=\frac{1}{2\pi} p_M^D(t, \pi u,\pi u).
\end{equation}
Therefore, by fiber integration with $\int_\sone\mathrm d\theta=2\pi$, \eqref{eq:global_index_X} becomes
\begin{equation}
\label{eq:index_s_one}
\ind_\sone (D_X)=\int_M \str p_M^D(t,x,x)\mathrm dx,\quad x=\pi u,\quad \forall t>0,
\end{equation}
which is of the form of the classical McKean-Singer formula on $M$.

In short, the $\sone$-equivariant index $\ind_\sone (D_X)$ above is connected to the ordinary index $\ind(D)$ on $M$. A similar statement is discussed in e.g. \cite[Theorem 4.7]{cheng2015heat}.

Denote by $I(x)=\lim_{t\downarrow 0^+}\str p_M^D(t,x,x)$ in \eqref{eq:index_s_one}. It is well known that $I(x)=\widehat A(TM)\wedge\ch\xi$, where $\widehat A(TM)$ and $\ch\xi$ are respectively the $\widehat A$-genus of the tangent bundle $TM$ and the Chern character of the complex vector bundle $\xi$ (e.g. \cite[Theorem 4.8]{berline2003heat}). This formula can be converted ``transversally'' into the desired index theorem on $X$ (with notations following those in Theorem \ref{thm:main_result}):
\begin{equation}
\ind_\sone (D_X)=\frac{p}{2\pi}\int_X \widehat A(\mathcal H)\wedge\ch\xi\wedge\omega_0\quad (\text{here }p=1).
\end{equation}

In this paper, we put emphasis on the manifold $X$ with $\sone$-action being locally free rather than globally free.

\section{Transversal Heat Kernels on $X$} 
\label{sec:const}

Henceforth $X$ is a compact, oriented manifold of dimension $n+1=2\ell +1$ with locally free $\sone$-action and we let $\pi:X\to M=X/\sone$. Here ``locally free'' means that the isotropy subgroup $H_u\subset\sone$ of $u$ is finite for every $u\in X$. If $H_u=\{\rm{id}\}$ for every $u\in X$, we say that the action is \textit{globally free} or free for short. For globally free $\sone$-action, the base $M$ is a smooth manifold; for locally free $\sone$-action, $M$ has a natural orbifold structure (cf. \cite[p. 173]{duistermaat2013heat}).

Remark that we mainly focus on the case where $\left\{ u\in X: \{\mathrm{id}\}=H_u\subset\sone\right\}\neq\emptyset$, as the general case is easily reduced to it (see Theorem \ref{thm:main_result}).

There exists an $\sone$-invariant metric on $X$. We normalize this metric so that the real vector field $T$ (cf. Remark \ref{rmk:smoothness}) induced by the $\sone$-action is of unit norm. Although $M$ now is an orbifold, the $\sone$-invariant metric also induces a Riemannian metric on $M$; see Subsection \ref{subsection: orbibundle} for more.

Denote by $\mathcal H\subset TX$ the (horizontal) subbundle of rank $n$ consisting of those tangent vectors orthogonal to $\sone$-orbits, and by $\mathcal H^\ast\subset T^\ast X$ its dual (via the $\sone$-invariant metric).

We assume that our manifold $X$ is \textit{transversally} spin and that the $\sone$-action on $X$ is of \textit{transversally even type} (cf. \cite[p. 295]{lawson2016spin}). Simply put, these assumptions guarantee that on $\mathcal H$ there exists a spin structure as a principal $\mathrm{Spin}(n)$-bundle $\mathscr{SP}(\mathcal H)$ and that the $\sone$-action can be lifted on $\mathscr{SP}(\mathcal H)$. See Appendix \ref{appendix:spin}.

Along the line of thought \eqref{eq: def of p_X_global} (for the globally free case), we shall get our heat kernel on $X$ by lifting the heat kernel on $M$. Let us start with the construction of the heat kernel on $M$.


\subsection{Construction of $\bigtriangleup_{\ortho{M}}^H$, $\bigtriangleup_X^H$ for an Orbifold ${M}$}
\label{subsection: orbibundle}

Suppose that the orbifold $M$ has admitted a Riemannian metric (see e.g. \cite[Proposition 2.20]{moerdijk2003introduction}). Given an orbifold chart $(\widetilde U,G,U,\pi)$ on $M$ (with $\pi:\widetilde U\to\widetilde U/G\cong U\subset M$), we have $\widetilde U\times\sone/G \cong \pi^{-1}(U)\subset X$ locally; see Appendix \ref{appendix:adaptation} for more. It is seen that the oriented orthonormal frame bundle $\ortho{M}$ is obtained by $\ortho{M}=\ortho{\mathcal H}/\sone$ where $\mathcal H$ is the horizontal subbundle of $TX$ as just indicated above. In fact, the orthonormal frame bundle $\ortho{M}$ is a smooth manifold (cf. \cite[Section 2.3]{gordon2012orbifolds}).

\begin{rmk}
\label{rmk:Laplace_M_OM}
In the notation of Subsection \ref{subsection:Notations and Set Up}, the Bochner's horizontal Laplacian $\Laplace^H_{\ortho{M}}$ on $\ortho{M}$ is defined by $\Laplace^H_{\ortho{M}}\coloneqq \sum_{i=1}^n H_i^2$ and satisfies $\Laplace^H_{\ortho{M}}\widehat \theta(\widehat x)=\widehat x^{-1}\Laplace_M\theta(x),\; \widehat \pi\widehat x=x$ (see \cite[p. 193]{hsu2002stochastic} for $\widehat{\theta}$). Similarly\footnote{The foundation of this analogy is described in Appendix \ref{appendix:adaptation}. See also  introductory paragraphs in Subsection \ref{subsection:transversal_D_X} for Laplacians associated with Dirac operators on $X$ and $M$.}, the \textit{transversal} or \textit{horizontal} Laplacian $\Laplace^H_X$ on $X$ is defined and satisfies
\begin{equation}
\label{eq:Lplace^H_X}
\Laplace^H_X\widetilde \theta(u)=\big(\pi^\ast(\Laplace_M\theta)\big)(u),\quad \pi u=x.
\end{equation}
\end{rmk}

\begin{nota}
\label{rmk:notation_p}
Let $p_M(t,x,y)$ denote the heat kernel associated with $\Laplace_M$ on functions over $M$ (cf. Proposition \ref{prop:heat_eq} below) and $P_M(t,x,y)$ (or simply $P_M(t)$) as the corresponding operator. Namely, $P_M(t)f(x)=\int_M p_M(t,x,y)f(y)\mathrm dy$.
\end{nota}

\subsection{Transversal Heat Kernel for $\sone$-invariant Functions on ${X}$}
\label{subsection:Transversal Heat Kernel}
The construction of heat kernels on the orbifold $M$ follows the approach by Dryden \textit{et al}. in \cite[Section 3]{dryden2008asymptotic}, in which {asymptotic} solutions are constructed locally on each orbifold chart, and patched up by using partition of unity. Out of this, with the standard successive approximation the fundamental solution $p_M(t,x,y)$ on $M$ is constructed. We summarize it:
\begin{prop}
\label{prop:heat_eq}
The heat kernel $p_M(t,x,y)$ on a compact orbifold $M$ (endowed with a Riemannian metric) which is the solution to 
the following heat equation can be constructed
\begin{equation}
\label{eq:heat_eq_M}
\begin{cases}
&\displaystyle\frac{\partial }{\partial t}p_M(t,x,y)=\frac{1}{2}\Laplace_Mp_M(t,x,y),\quad (t,x,y)\in (0,\infty)\times M\times M,\\
&\displaystyle\lim_{t\downarrow 0^+}\int_M p_M(t,x,y) f(y)\mathrm dy=f(x),\quad f\in C^\infty(M).
\end{cases}
\end{equation}
\end{prop}

Recall $\Laplace_X^H$ of \eqref{eq:Lplace^H_X}. Our goal now is to construct the heat kernel $p_X(t,u,v)$ that satisfies the following \textit{transversal} heat equation on $X$:
\begin{equation}
\label{eq:heat_eq_on_X}
\begin{cases}
&\displaystyle\frac{\partial }{\partial t}p_X(t,u,v)=\frac{1}{2}\Laplace^H_Xp_X(t,u,v),\quad (t,u,v)\in (0,\infty)\times X\times X, \\
&\displaystyle\lim_{t\downarrow 0^+}\int_Xp_X(t,u,v)\omega(v)\mathrm dv=\mathcal{P}_0\omega(u),\quad \omega\in C^\infty (X).
\end{cases}
\end{equation}
The operator $\mathcal P_0$ in \eqref{eq:heat_eq_on_X} above is given as follows.

\begin{defn}
\label{defn:projection}
The projection operator $\mathcal P_0$ on $\omega\in\Omega(X)\big(\supset C^\infty(X)\big)$, the space of smooth differential forms on X, is defined by
\begin{equation*}
\mathcal{P}_0\omega(v) =\frac{1}{2\pi}\int_{\sone}(e^{-i\xi})^\ast\omega (v)\mathrm d\xi.
\end{equation*}
Here $(e^{-i\xi})^\ast\omega (v)$ means the pull-back $\big((e^{-i\xi})^\ast\omega\big)(v)$ by the action $e^{-i\xi}$ on $X$.
\end{defn}

We are going to construct $p_X(t,u,v)$ of \eqref{eq:heat_eq_on_X} using the heat kernel $p_M(t,x,y)$ of \eqref{eq:heat_eq_M}.\comment{According to \cite[Section 3]{dryden2008asymptotic} the heat kernel $p_M(t,x,y)$ (on functions) satisfying the following equation, is constructed.} We claim that by defining
\begin{equation}
\label{def:p_X}
p_X(t,u,v)\coloneqq \frac{1}{2\pi}p_M(t,\pi u,\pi v),
\end{equation}
the transversal heat kernel $p_X(t,u,v)$ is a solution to the {transversal} heat equation \eqref{eq:heat_eq_on_X}.

\begin{prop}(Existence)
\label{prop:existence_p_X}
$p_X(t,u,v)$ of \eqref{def:p_X} satisfies the transversal heat equation \eqref{eq:heat_eq_on_X}.
\end{prop}
\begin{proof}
The first equation of \eqref{eq:heat_eq_on_X} follows from \eqref{eq:Lplace^H_X}, \eqref{eq:heat_eq_M}, \eqref{def:p_X} and our normalization of the $\sone$-invariant metric along the $\sone$-orbit. For the second equation, the verification is straightforward and is omitted.
\end{proof}

The uniqueness of the heat kernel (see Theorem \ref{thm:uniqueness}) is essential for the index theorem. To reach it, some preparations are in order. These are going to be useful in Subsection \ref{subsection:transversal_D_X}. For another important use, see \eqref{eq:kernel computation_integral} of Subsection \ref{subsection:Prob} and the paragraph at the end of this paper.

To start with, let us introduce the notion for the $m$-th Fourier component. Note that $\mathcal{P}_0\omega = \omega_0$ (cf. Definition \ref{defn:projection}) can be seen as the $m$-th Fourier component with $m=0$. We set up the following.

\begin{defn}
\label{def:Fourier_component_X}
For every $m\in\mathbb{Z}$ define 
\begin{equation}
\Omega_m (X)\coloneqq \left\{ f\in\Omega (X): (e^{-i\zeta})^\ast f = e^{-im\zeta}f,\quad \forall \zeta\in [0,2\pi) \right\}
\end{equation}
as the space of all $m$-th Fourier components where $(e^{-i\zeta})^\ast$ is the pull-back map.
\end{defn}
Some basic facts are well known.
\begin{lemma}
\label{lemma:Fourier_decomposition_X}
Every $\widetilde \theta\in\Omega(X)$ can be decomposed into series of $\{ \widetilde\theta_m\in\Omega_m(X) \}_{m\in\mathbb{Z}}$ in the $L^2$-sense. 
\end{lemma}

\begin{proof}
It is easy to verify that $\widetilde\theta=\sum_{m\in\mathbb{Z}}\widetilde\theta_m$ where $\displaystyle\widetilde\theta_m(u)\coloneqq \frac{1}{2\pi}\int_\sone (e^{-i\zeta})^\ast\widetilde\theta (u)e^{im\zeta}\mathrm d\zeta$.
\end{proof}

There is the following orthogonality between $\Omega_0(X)$ and $\Omega_m(X)$ for $m\neq 0$, whose verification is straightforward and will be omitted.
\begin{prop}
\label{prop:ortho}
We have $\Omega_0(X)\perp \Omega_m(X)$ for $m\neq 0$. The operator $P_X(t)$ (cf. Notation \ref{rmk:notation_p}) annihilates the non-$\sone$-invariant part of $f\in\Omega(X)$, i.e., $P_X(t)f_m=0$ for $f_m\in\Omega_m(X)$ with $m\neq 0$, and $P_X(t)\big(\Omega(X)\big)\subset\Omega_0(X)$.
\end{prop}

\subsection{Transversal Heat Kernel $p_{X}^{D}(t,u,v)$ for $\Box_{X}^{D}$ on  ${X}$}
\label{subsection:transversal_D_X}
Most treatment here is parallel to the preceding subsection. Let $X$ be as in the beginning of Section \ref{sec:const}. A \textit{transversal} or \textit{horizontal} Dirac operator can be constructed in a way parallel to \eqref{eq:Dirac_composition}:
\begin{equation}
\label{eq:Dirac_composition_new}
D_{X}:\Omega(G)\overset{\nabla}{\longrightarrow}\Omega(\mathcal H^\ast\otimes G)\overset{\text{dual}}{\longrightarrow}\Omega(\mathcal H\otimes G)\overset{c}{\longrightarrow}\Omega(G),
\end{equation}
where $G$, with an $\sone$-equivariant complex vector bundle $\xi$ with an $\sone$-invariant connection $\nabla^{\xi}$, 
is analogously defined. Denote by $D_X^\pm$ the operator $D_X:\Omega(G^\pm)\to \Omega(G^\mp)$ (cf. \eqref{eq:Dirac_composition}) and $\DLaplace_X^D\coloneqq -D_X^2$.

The assumption that the $\sone$-action on $X$ is of transversally even type allows us to have a Dirac operator $D_M$ on $M$ induced by the above transversal Dirac operator $D_X$ on $X$ (cf. \cite[Subsection 2.4, p. 629]{lott2000signatures}), which naturally leads to a compatibility relation\footnote{Note that the set $S$ of orbifold points in $M$ (i.e. the union of lower-dimensional strata) presents no serious problem here, because the compatibility relation naturally seen on $M\setminus S$ extends across $S$ to the whole $M$ by using continuity, as those operators are globally defined on $X$ and $M$.} between $-D_X^2$ and $-D_M^2$ (via $\pi: X\to M$) analogous to the relation \eqref{eq:Lplace^H_X}.

\begin{nota} (cf. Notation \ref{rmk:notation_p})
\label{rmk:notation_p^D}
Denote by $p_M^D(t,x,y)$ the heat kernel associated with $-D_M^2=:\DLaplace_M^D$ on the twisted bundle $G$, and $P_M^D(t,x,y)$ (or simply $P_M^D(t)$) as the corresponding operator. Here $G$, in the same notation as in \eqref{eq:Dirac_composition_new} on $X$, denotes the naturally induced bundle on $M$.
\end{nota}

The existence of $p_M^D(t,x,y)$ above is well known (see e.g. \cite[Chapter 14]{duistermaat2013heat}).\comment{, which we will use to prove the existence of $p_X^D(t,u,v)$, the heat kernel of the $\sone$-transversally elliptic operator $P_X^D(t,u,v)$ (cf. \cite[Section 2.2]{paradan2009index} for more general treatment).

} Given $p_M^D(t,x,y)$, the transversal heat kernel $p_X^D(t,u,v)$ and the corresponding operator $P_X^D(t)$ can be obtained by lifting $p_M^D(t,x,y)$ to $X$ in a way similar to \eqref{def:p_X}. We conclude with the following result (cf. Proposition \ref{prop:existence_p_X}).

\begin{prop}
\label{prop:heat_eq_D_X}
With the assumptions and notations above, the transversal heat kernel $p^D_X(t,u,v)$ for $-D_X^2=:\DLaplace_X^D$ satisfying
\begin{equation}
\label{eq:heat_eq_for_p_X^D}
\begin{cases}
&\displaystyle\frac{\partial }{\partial t}p^D_X(t,u,v)=\frac{1}{2}\DLaplace_X^D\,p^D_X(t,u,v),\quad (t,u,v)\in(0,\infty)\times X\times X\\
&\displaystyle\lim_{t\downarrow 0^+}\int_Xp^D_X(t,u,v)\omega(v)\mathrm dv=\mathcal{P}_0^D\omega(u),\quad \omega\in\Omega(G),
\end{cases}
\end{equation}
can be constructed, where $\mathcal P_0^D$ denotes the projection (defined in a way similar to Definition \ref{defn:projection}).
\end{prop}

For our use later on, here are some results corresponding to those in Subsection \ref{subsection:Transversal Heat Kernel}.

\begin{defn}
\label{def:Fourier_component_G}
For $m\in\mathbb{Z}$ define $\Omega_m(G)\coloneqq \{ f\in\Omega(G):(e^{-i\zeta})^\ast f=e^{-im\zeta}f,\quad\forall \zeta\in [0,2\pi)\}$ as the space of all $m$-th Fourier components.
\end{defn}

\begin{prop}
\label{prop:ortho_G}
Every $\widetilde\theta\in\Omega(G)$ can be decomposed into series of $\{\widetilde\theta_m\in\Omega_m(G)\}_{m\in\mathbb{Z}}$ in the $L^2$-sense. We have $\Omega_0(G)\perp\Omega_m(G)$ for $m\neq 0$. The operator $P_X^D(t)$ annihilates the non-$\sone$-invariant part of $f\in\Omega(G)$, i.e., $P_X^D(t)f_m=0$ for $f_m\in\Omega_m(G)$ with $m\neq 0$, and $P_X^D(t)\big(\Omega(G)\big)\subset\Omega_0(G)$.
\end{prop}

Consider $D_{X}^{\pm}\big|_{\Omega_0(G^\pm)}:\Omega_0(G^\pm)\to\Omega_0(G^\mp)$ (still denoted by $D_X^\pm$ here) with $\DLaplace_X^{D,\pm}\coloneqq -(D_X^\pm)^2$. The corresponding transversal heat kernels are denoted by $p_X^{D,\pm}(t,u,v)$ respectively (cf. Notation \ref{rmk:notation_p^D}).

\begin{nota}
\label{nota:projection_D_pm}
Denote by $P_X^{D,\pm}(t)$ the operators associated with $p_X^{D,\pm}(t,u,v)$ (cf. Notation \ref{rmk:notation_p^D}) and by $\mathcal P_0^{D, \pm}$ the corresponding projections (cf. Definition \ref{defn:projection}).
\end{nota}

\comment{Likewise, the above results also hold true for (transversal) heat kernels $p_{X,0}^{D,\pm}(t,u,v)$ for $\DLaplace_{X,0}^{D,\pm}\coloneqq -(D_{X,0}^{\pm})^2$ on $X$. We denote the corresponding projections by $\mathcal P_0^{D,\pm}$ (cf. \eqref{eq:heat_eq_for_p_X^D}).}

We are now ready to focus on the uniqueness of the heat kernel $p_{X}^{D,+}(t,u,v)$; the counterpart for $p_X^{D,-}(t,u,v)$ can be done in parallel. See \cite[Theorem 5.14]{cheng2015heat} for the uniqueness in the context of CR manifolds with $\sone$-action.


\begin{thm}(Uniqueness)
\label{thm:uniqueness}
The transversal heat kernel $p_{X}^{D,+}(t,u,v)$ is unique (in the sense of Remark \ref{rmk:uniqueness_condition} below). In particular, this implies that the associated operator $P_X^{D,+}(t)$ is self-adjoint on $X$.
\end{thm}
\begin{proof}
The idea follows that of the standard arguments \cite[Lemma 2.16]{berline2003heat}, although we are in the transversal case. But since the self-adjointness of $P_X^{D,+}(t)$ holds as an outcome, let us treat it in some detail. Compare with the operator $e^{\frac{1}{2}t\DLaplace_{X,0}^{D,+}}$ defined in \eqref{eq:heat_operator_expansion} , which satisfies \eqref{eq:eq_for_e_box}.
Then, for any $f,g\in\Omega(G^+)$
\begin{align*}
0&=-\int_0^t \frac{\partial}{\partial s}\langle P_{X}^{D,+}(t-s)f,e^{\frac{1}{2}s\DLaplace_{X,0}^{D,+}}g\rangle \mathrm ds=\langle P_{X}^{D,+}(t)f,\mathcal{P}_0^{D,+}g\rangle-\langle \mathcal{P}_0^{D,+}f,e^{\frac{1}{2}t\DLaplace_{X,0}^{D,+}}g\rangle \\
&=\langle P_{X}^{D,+}(t)f, g\rangle-\langle f,e^{\frac{1}{2}t\DLaplace_{X,0}^{D,+}}g\rangle \text{ (by Proposition \ref{prop:ortho_G} and \eqref{eq:heat_operator_expansion} below)}\\
&=\langle P_{X}^{D,+}(t)f,g\rangle-\langle e^{\frac{1}{2}t\DLaplace_{X,0}^{D,+}}f,g\rangle\text{ (since $e^{\frac{1}{2}t\DLaplace_{X,0}^{D,+}}$ is self-adjoint).}
\end{align*}
This implies $e^{\frac{1}{2}t\DLaplace_{X,0}^{D,+}}=P_{X}^{D,+}(t)$, giving the self-adjointness  of the latter since the former is self-adjoint.
\end{proof}

\begin{rmk}
\label{rmk:uniqueness_condition}
Other than $e^{\frac{1}{2}t\DLaplace_{X,0}^{D,+}}$ of \eqref{eq:heat_operator_expansion}, this argument is applicable to any self-adjoint operator $H_t$ which satisfies Proposition \ref{prop:heat_eq_D_X} and annihilates $\Omega_m(G),$ $\forall m\neq 0$.
\end{rmk}


\section{Probabilistic Approach for our Index Theorem on $X$}
\label{sec:prob}

Classically the computation of the local index density relies in part on the asymptotic expansion of the appropriate heat kernel. In our transversal setting, one difficulty arising in this context is that the asymptotic expansion of the heat kernel for $\DLaplace_X^D$ on $X$ appears vague, partly because the $\sone$-action on $X$ is only locally free. See, however, the work \cite{cheng2015heat} which discusses the asymptotic expansion for CR manifolds with $\sone$-action (\cite[Theorem 1.3 and Subsection 7.1]{cheng2015heat}).  This unavoidably contains nontrivial ``corrections" in contrast to the globally 
free case and their approach relies on the CR geometry consideration, cf. Introduction.

The difficulty mentioned above in the purely Riemannian setting here is overcome in a way different from \cite{cheng2015heat}. This is to adopt the probabilistic approach with the fact that the desired heat kernel can be represented in terms of \textit{conditional expectation} together with the transversal heat kernel for \textit{functions} (see \eqref{eq: kernel_represent}). The asymptotic expansion of the latter is more tractable. Further, as far as the ``supertrace" is 
concerned, with the probabilistic approach the corrections mentioned above are hidden and implicitly cancelled off in 
the end.   Those corrections needn't enter our picture here (for the purpose of supertrace).  

The discussion is separated into two parts: the probabilistic representation of the heat kernel for $-D_X^2$ on $X$, and the computation of the local index density.


\subsection{Probabilistic Representation of Heat Kernels}
\label{subsection:Prob}
To obtain a  Feynman-Kac formula that connects the transversal heat kernel $p_X^D(t,u,v)$ of Proposition \ref{prop:heat_eq_D_X} with the transversal heat kernel $p_X(t,x,y)$ (on functions) is not new.  Yet the 
transversal condition here leads to a variant of the classically known one.  Our main results are \eqref{eq: kernel_represent} and \eqref{eq:diag_kernel} below.

Let the notation and assumption be as in the beginning of Section \ref{sec:const}. Let $\pi:X\to M=X/\sone$ and $\widehat \pi:\ortho{M}\to M$. 
Consider the following for $\widetilde\theta\in\Omega_0(G)$:
\begin{equation}
\label{eq:heat_eq_D_X}
\begin{cases}
&\displaystyle\frac{\partial \widetilde{\theta}}{\partial t}(t,u)=\frac{1}{2}\DLaplace_X^D\widetilde{\theta}(t,u),\quad (t,u)\in (0,\infty)\times X\\
&\displaystyle\widetilde{\theta}(0,u)=\widetilde{\theta}_0(u),\quad \widetilde\theta_0\in\Omega_0(G)\text{ of Definition \ref{def:Fourier_component_G}}.
\end{cases}
\end{equation}
It follows that (cf. Proposition \ref{prop:ortho_G})
\begin{equation}
\label{eq: kernel_prop}
\widetilde\theta(t,u)=\int_Xp_X^D(t,u,v)\widetilde\theta_0(v)\mathrm dv.
\end{equation}

In this subsection, we aim to develop a probabilistic representation of $p_X^D(t,u,v)$.

Fix $x\in M$ and let $\{X_i\}$ be an orthonormal basis of $T_xM$ (the orbifold tangent bundle cf. \cite[Section 2]{wolak2015orbifolds} or as identified with the horizontal subbundle $\mathcal H$ on $X$ modulo action of $\sone$).  
Identify $T_xM$ with $\mathbb{R}^n$ via this frame $\{X_i\}$ and denote it by $\widehat x:\mathbb{R}^n\to T_xM $ (cf. Remark \ref{rmk:Hodge-de Rham_Laplace}).

Note that the preceding statements make perfect sense if $x$ lies in the principal stratum of $M$ (i.e. the set of points where the local group associated with orbifold charts is trivial). For $x$ in the lower-dimensional strata, those statements still make 
sense by going to the orbifold charts.  The modifications (here and below) are not difficult and omitted. See Appendices \ref{appendix:L_diffusion} and \ref{appendix:adaptation}.

The usual Lichnerowicz formula (cf. \cite[Theorem 7.5.3]{hsu2002stochastic}) that expresses $\DLaplace_M^D-\Laplace_M$ as $ -{S}/{4}-{1}/{2}\sum_{j,k=1}^n c(X_j)c(X_k)\otimes L(X_j,X_k)$, where $\DLaplace_M^D=-D_M^2$ for the Dirac operator $D_M$ of Subsection \ref{subsection:transversal_D_X} and $\Laplace_M$ in Remark \ref{rmk:Laplace_M_OM}, $S$ the scalar curvature and $L$ the curvature operator on the factor bundle $\xi$ of $G$ (cf. \eqref{eq:Dirac_composition}), can be lifted up to $\ortho{M}$ (cf. \cite[p. 197--198]{hsu2002stochastic}):
\begin{equation}
\label{eq:Lichnerowicz_lifted}
\DLaplace_\ortho{M}^D=\Laplace^H_{\ortho{M}} -\frac{\widehat S}{4}-\frac{1}{2}\sum_{j,k=1}^n c(\widehat X_j)c(\widehat X_k)\otimes \widehat L(\widehat X_j,\widehat X_k),
\end{equation}
where $\DLaplace_\ortho{M}^D$ lifts $\DLaplace_M^D$ and satisfies $\DLaplace_\ortho{M}^D\widehat\theta(\widehat x)=\widehat x^{-1}\DLaplace_M^D\theta(x)$ ($\widehat \pi\widehat x=x$, cf. Remark \ref{rmk:Laplace_M_OM}).

Now considering the following equation for $\widehat\theta$ (where $\widehat\theta_0(\widehat x)=\widehat x^{-1}\theta_0(x)$, $\pi^\ast\theta_0(x)=\widetilde\theta_0(u)$ in \eqref{eq:heat_eq_D_X} with $\pi u=\widehat\pi \widehat x$):
\begin{equation}
\label{eq:heat_eq_O(M)}
\displaystyle\frac{\partial \widehat{\theta}}{\partial t}(t,\widehat x)=\frac{1}{2}\DLaplace_\ortho{M}^D\widehat{\theta}(t,\widehat x),\quad (t,\widehat x)\in (0,\infty)\times \ortho{M};\quad\widehat{\theta}(0,\widehat x)=\widehat{\theta}_0(\widehat x),
\end{equation}
the probabilistic representation of the solution $\widehat \theta(t,\widehat x)$ to \eqref{eq:heat_eq_O(M)} can be derived via an appropriate Feynman-Kac formula (cf. Theorem \ref{thm:F-K_formula}) incorporating with the Lichnerowicz formula \eqref{eq:Lichnerowicz_lifted}:  Let $M_t:G_x\rightarrow G_x$ be determined by 
${\mathrm dM_t}/{\mathrm dt}=-{1}/{4}M_t\sum_{j,k=1}^n c(e_j)c(e_k)\otimes \widehat X_t L(\widehat X_te_j,\widehat X_te_k)\widehat X_t^{-1},\; M_0=I$, where $\widehat X_t$ denotes the horizontal lift of an $M$-valued Brownian motion $X_t$ 
and set $R_t=\expo\big(-{1}/{8}\int_0^t\widehat S(\widehat X_s)\mathrm ds\big)$ (cf. \cite[Proposition 7.6.4; p. 82]{hsu2002stochastic}).  The solution to the initial value problem \eqref{eq:heat_eq_O(M)} on $\ortho{M}$ is then given by $\widehat{\theta}(t,\widehat x)=\mathbb{E}_{\widehat x}\big[ R_t M_t\widehat{\theta}_0(\widehat X_t)\big]$ (cf. Theorem \ref{thm:F-K_formula} and Remark \ref{rmk: FK}).  

Correspondingly the solution to the initial value problem on $M$ (where $\theta_0(x)$ satisfies $\widehat\theta_0(\widehat x)=\widehat x^{-1}\theta_0(x)$, $\widehat \pi\widehat x=x$): ${\partial {\theta}}(t,x)/{\partial t}={1}/{2}\DLaplace_M^D{\theta}(t,x),\; (t,x)\in (0,\infty)\times {M};\; \theta(0,x)={\theta}_0(x)$ is given, via $\widehat\theta_0(\widehat x)=\widehat x^{-1}\theta_0(x)$, by ${\theta}(t, x)=\mathbb{E}_{ x}\big[ R_t M_t\widehat X_t^{-1}{\theta}_0( X_t)\big]$. See \cite[Theorem 7.2.1, Proposition 7.6.4]{hsu2002stochastic}.

Remark that $\widehat{X}_t^{-1}$ may be interpreted as the stochastic parallel transport $\tau_t$ \cite[p. 200 and p. 219]{hsu2002stochastic}, \cite[(2.4) in p. 97]{hsu1988brownian}. See also Remark \ref{rmk:111} in Subsection \ref{subsection:Estimation of Horizontal Brownian Motion} below. 

An adaptation of the above to orbifolds is partly discussed in Appendices \ref{appendix:L_diffusion} and \ref{appendix:adaptation}.
We will keep using those adaptations without further mention. 

In complete analogy\footnote{For a second order elliptic operator $L$ on $X$ (not necessarily nondegenerate), the $L$-diffusion process exists (\cite[Theorem 1.3.4]{hsu2002stochastic}). Thus, by setting $L=\DLaplace_X^D/2$ and by arguments similar to \cite[p. 199 and p. 24]{hsu2002stochastic} one has the formula \eqref{eq:F-K_X}. For our case that $\widetilde\theta_0$ is bundle-valued rather than number-valued, one may consult \cite[Subsection 2c, p. 73]{bismut1984atiyah} or \cite[Section 1]{hsu1988brownian}, \cite[Remark 7.2.5]{hsu2002stochastic}. Remark that the local freeness of the $\sone$-action on $X$ does not introduce extra difficulties into the aforementioned methods. Alternatively, since $p_X^D(t,u,v)$ and $p_X(t,u,v)$ have been constructed (with uniqueness) in Subsection \ref{subsection:Estimation of Horizontal Brownian Motion}, combining arguments similar to \cite[Proposition 4.1.3, Theorem 1.3.6 and p. 199]{hsu2002stochastic} and \cite[Section 2]{hsu1988brownian} yields the $\DLaplace_X^D/2$-diffusion process on $X$ satisfying \eqref{eq:F-K_X}.}, it follows from $\widetilde\theta_0(u)=\theta_0(\pi u)$ that the solution $\widetilde\theta$ to \eqref{eq:heat_eq_D_X} on $X$ is ${\widetilde\pi}^\ast\theta$ and can be represented as 
\begin{equation}
\label{eq:F-K_X}
\widetilde{\theta}(t, u)=\mathbb{E}_{ u}\big[ R_t M_t\widehat X_t^{-1}\widetilde{\theta}_0(U_t)\big];
\end{equation}
see Remark \ref{rmk:Brownian_bridge} for the needed semimartingale property of $U_t$.
Here, by abuse of notation we keep $\widehat X_t^{-1}$ rather than $\tau_t$ in the RHS of \eqref{eq:F-K_X}.

Combining \eqref{eq:F-K_X} and \eqref{eq: kernel_prop}, we have
\begin{equation}
\label{eq:compare_E_and_p}
\mathbb{E}_u\big[ R_tM_t \widehat X_t^{-1}\widetilde{\theta}_0(U_t)\big]=\int_X p_X^D (t,u,v)\widetilde{\theta}_0(v)\mathrm dv.
\end{equation}

The equality \eqref{eq:compare_E_and_p} holds only for those $\widetilde{\theta}_0$ (i.e. $\sone$-invariant ones). We need to generalize the result \eqref{eq:compare_E_and_p} for arbitrary $\widetilde \theta$, cf. \eqref{eq:diag_kernel} below.

Let $\widetilde \theta\in\Omega(G)$. From the orthogonality of $\Omega_m(G)$ in Proposition \ref{prop:ortho_G},
\begin{equation}
\label{eq:kernel computation_integral}
\int_Xp_X^D (t,u,v) \widetilde\theta(v)\mathrm dv
=\int_Xp_X^D (t,u,v) \widetilde\theta_0(v)\mathrm dv\overset{\eqref{eq:compare_E_and_p}}{=\joinrel=}\mathbb{E}_u\big[ R_t M_t\widehat X_t^{-1}\widetilde\theta_0(U_t)\big].
\end{equation}
We compute, with $\widetilde\theta_0(U_t)=\widetilde\theta_0(e^{-i\xi}\circ U_t)$,
\begin{align}
&\mathbb{E}_u\big[ R_t M_t\widehat X_t^{-1}\widetilde\theta_0(U_t)\big]=\mathbb{E}_u\Big[ R_t M_t\cdot \frac{1}{2\pi}\int_{\sone}\widehat X_t^{-1}\widetilde{\theta}_0(e^{-i\xi}\circ U_t)\mathrm d\xi\Big]\nonumber\\
=&\frac{1}{2\pi}\int_{\sone}\int_X  \mathbb{E}_{u,v;t}\big[ R_t M_t\widehat X_t^{-1}\widetilde{\theta}_0(e^{-i\xi}\circ v) \big]\cdot p_X(t,u,v)\mathrm dv\mathrm d\xi \label{eq:kernel computation_expectation2}\quad\text{(see remarks below)}\\
\overset{v=e^{i\xi}\circ \widetilde v}{=\joinrel=}&\frac{1}{2\pi}\int_X\int_{\sone}  \mathbb{E}_{u,e^{i\xi}\circ\widetilde v;t}\big[ R_t M_t\widehat X_t^{-1}\widetilde{\theta}_0(\widetilde v) \big]\cdot p_X(t,u,e^{i\xi}\circ\widetilde v)\mathrm d\xi \mathrm d\widetilde v.\label{eq:kernel computation_expectation}
\end{align}
For \eqref{eq:kernel computation_expectation2} above, see Subsection \ref{subsection:Some Results in Probability}. In the present context with the horizontal Laplacian $\Laplace_X^H$ of \eqref{eq:Lplace^H_X} and $\DLaplace_X^D$ of \eqref{eq:heat_eq_for_p_X^D}, we refer to \cite[Proposition 4.1.3 in p. 104, (7.2.5) in p. 200 and the last paragraph in p. 219]{hsu2002stochastic}; the horizontal/transversal property of our operator here does not cause serious difficulties.  It is somewhat novel in that the above prefers to work out on $X$ rather than on (the orbifold) $M$.

Hence, comparing \eqref{eq:kernel computation_integral} and \eqref{eq:kernel computation_expectation} we achieve our probabilistic representation:
\begin{equation}
\label{eq: kernel_represent}
p_X^D (t,u,v)=\frac{1}{2\pi} \int_{\sone}\mathbb{E}_{u,e^{i\xi}\circ v;t}\big[ R_t M_t\widehat X_t^{-1}\mathcal P_0^D\big]p_X(t,u,e^{i\xi}\circ v)\mathrm d\xi
\end{equation}
as mentioned earlier, where the projection $\mathcal P_0^D$ is as in Proposition \ref{prop:heat_eq_D_X}.  In particular\footnote{To be slightly more precise about \eqref{eq:diag_kernel}, we insert into both sides of \eqref{eq: kernel_represent} a $\delta$-function like section so that $p_X^D(t,u,u)=\frac{1}{2\pi}\int_{v\in X}\int_\sone\mathbb{E}_{u,e^{i\xi}\circ v;t}\left[R_tM_t\widehat X_t^{-1}\big(\mathcal P_0^D\delta(u-v)\big)\right]p_X(t,u,e^{i\xi}\circ v)\mathrm d\xi\mathrm dv$ with $\mathcal P_0^D$ explicated in Notation \ref{nota:projection_D_pm} and Definition \ref{defn:projection}.}
\begin{equation}
\label{eq:diag_kernel}
p_X^D (t,u,u)=\frac{1}{2\pi} \int_{\sone}\mathbb{E}_{u,e^{i\xi}\circ u;t}\big[ R_t M_t\widehat X_t^{-1}\mathcal P_0^D\big]p_X(t,u,e^{i\xi}\circ u)\mathrm d\xi.
\end{equation}
See \eqref{eq:before_exchange} for use. In contrast to the classical case, a projection $\mathcal P_0^D$ is introduced into the conditional expectation.

\subsection{A Transversal, Equivariant Atiyah-Singer Index Theorem}

We shall now see that with the \textit{interchangeability} of limits, our index theorem for the manifold $X$ with locally free $\sone$-action can be immediately derived. The interchageability will be discussed in Section \ref{sec: est_HBM}. 

Let us first aim at a McKean-Singer formula for our index problem. See Theorem \ref{thm:index_thm_str}.

\begin{defn}
\label{def:index}
Recall $\Omega_m(G)$ with $m=0$ from Definition \ref{def:Fourier_component_G} and $\mathcal P_0^{D,\pm}$ from Notation \ref{nota:projection_D_pm}. Set $D_{X,0}^\pm\coloneqq D_X^\pm\big|_{\Omega_0(G^\pm)}:\Omega_0(G^\pm)\to\Omega_0(G^\mp)$. Our transversal, $\sone$-equivariant index is defined to be the index of $D_{X,0}$. \comment{(cf. Subsection \ref{subsection:transversal_D_X}) is defined as}Namely $\ind_{\sone}(D_{X,0})\coloneqq \dimension \ker D_{X,0}^+-\dimension \ker D_{X,0}^- $, provided that both $\ker D_{X,0}^+$ and $\ker D_{X,0}^-$ are finite-dimensional.
\end{defn}

Let $0=\lambda_0\le \lambda_1\le \lambda_2\le \cdots$ be the eigenvalues of $-\DLaplace_{X,0}^{D,+}\coloneqq (D_{X,0}^+)^2$ counted with multiplicity (cf. Remark \ref{rmk:smoothness}). Denote $\{f_1^\lambda,\cdots,f_{d_\lambda}^\lambda\}$ as an orthonormal basis for the eigenspace of $-\DLaplace_{X,0}^{D,+}$ with eigenvalue $\lambda$. Form (cf. Remark \ref{rmk:smoothness} below)
\begin{equation}
\label{eq:heat_operator_expansion}
e^{\frac{1}{2} t\DLaplace_{X,0}^{D,+}}(u,v)=\sum_{\lambda\in\mathrm{Spec}\,(-\DLaplace_{X,0}^{D,+})}\sum_{i=1}^{d_\lambda} e^{-\lambda t/2}f_i^\lambda(u)\otimes \big(f_i^\lambda (v)\big)^\ast, 
\end{equation}
which 
naturally extends as $e^{\frac{1}{2} t\DLaplace_{X,0}^{D,+}}(u,v):\Omega(G^+)\to \Omega_0(G^+)\subset\Omega(G^+)$.  It satisfies the (operator) equation
\begin{equation}
\label{eq:eq_for_e_box}
\frac{\partial e^{\frac{1}{2} t\DLaplace_{X,0}^{D,+}}}{\partial t}=\frac{1}{2}\DLaplace_{X,0}^{D,+} e^{\frac{1}{2} t\DLaplace_{X,0}^{D,+}},\quad e^{\frac{1}{2} t\DLaplace_{X,0}^{D,+}}|_{t=0}=\mathcal P_0^{D,+}.
\end{equation}
Similar consideration applies to $-\DLaplace_{X,0}^{D,-}$.

\begin{rmk} (Smoothness)
\label{rmk:smoothness}
Let $T$ be the real vector field induced by the $\sone$-action, i.e. $T\omega=\frac{\partial}{\partial\zeta}(e^{-i\zeta})^\ast\omega( u)\big|_{\zeta=0}$.  Since the eigenfunctions of $-\DLaplace_{X,0}^{D,\pm}$ are still eigenfunctions of the elliptic operator $-(\DLaplace_{X}^{D,\pm}+T^2)$ restricted to $\Omega_0(G^\pm)$ as $T\big|_{\Omega_0(G^\pm)}=0$,
this leads to the smoothness of \eqref{eq:heat_operator_expansion} (cf. \cite[Lemmas 1.6.3 and 1.6.5]{gilkey1995invariant} and \cite[p. 45]{cheng2015heat}).
\end{rmk}

\begin{defn}
\label{nota:omega0}
(cf. \cite[p. 33]{cheng2015heat}) Let $\omega_0$ be the global real $1$-form on $X$ determined by $\langle \omega_0,\mathfrak h\rangle=0$ for all $\mathfrak h\in \mathcal H$ and $\langle\omega_0,T\rangle=1$ for $T$ in the remark above.
\end{defn}
Recall $\str e^{\frac{1}{2}t\DLaplace_{X,0}^D}=\trace e^{\frac{1}{2}t\DLaplace_{X,0}^{D,+}}-\trace e^{\frac{1}{2}t\DLaplace_{X,0}^{D,-}}$. Adapting Hsu's argument \cite[Theorem 7.3.1]{hsu2002stochastic} via \eqref{eq:heat_operator_expansion}, we have a similar formula (cf. \cite[Corollary 4.8]{cheng2015heat} in the CR context):
\begin{thm}(McKean-Singer formula)
\label{thm:index_thm_str}
For any $t>0$, 
\begin{equation}
\label{eq:McKean-Singer formula}
\ind_{\sone}(D_{X,0})=\int_X \str e^{\frac{1}{2}t\DLaplace_{X,0}^D}(u,u)\mathrm du.
\end{equation}
\end{thm}

This supertrace is studied in \cite{cheng2015heat} in the context of CR manifolds as remarked in the Introduction; 
see also \cite{barilari2012small} for heat kernels on sub-Riemannian geometry but no supertrace is discussed there.  

We shall now see that our main result Theorem \ref{thm:main_result} can be derived provided that the interchangeability of limits holds (cf. \eqref{assump:interchange}).

From \eqref{eq:diag_kernel} and Theorem \ref{thm:index_thm_str}, we have
\begin{align}
\label{eq:before_exchange}
I(t,u)&=\frac{1}{2\pi}\str \int_\sone\mathbb{E}_{u,e^{i\xi}\circ u;t}\big[R_tM_t\widehat X^{-1}_t\mathcal P_0^D\big]p_X(t,u,e^{i\xi}\circ u)\mathrm d\xi\nonumber \\ 
&=\frac{1}{2\pi}\int_\sone\mathbb{E}_{u,e^{i\xi}\circ u;t}\big[ R_t \,\str M_t\widehat X^{-1}_t\mathcal P_0^D\big]p_X(t,u,e^{i\xi}\circ u)\mathrm d\xi
\end{align}

\comment{Since $p_X(t,u,u)\sim (2\pi t)^{-n/2}$ and $R_t\to 1$ as $t\downarrow 0^+$. }Assume that the limit in $t$ is interchangeable with the integral, i.e., the RHS of \eqref{eq:before_exchange}$\big|_{t\downarrow 0^+}$ equals
\begin{equation}
\label{assump:interchange}
\frac{1}{2\pi}\int_\sone \mathbb{E}_{u,e^{i\xi}\circ u;t} \Big[ \lim_{t\downarrow 0^+}p_X(t,u,e^{i\xi}\circ u) R_t\str M_t \widehat X^{-1}_t \mathcal P_0^D\Big]\mathrm d\xi.
\end{equation}
Write $I(u)\coloneqq \lim_{t\downarrow 0^+}I(t,u)$ for \eqref{assump:interchange}.

By noting $p_X(t,u,e^{i\xi}\circ u)=\frac{1}{2\pi}p_M(t,x,x)$ for $u\in X$, one sees ($X_0$ denoting the principal stratum of $X$)
\begin{equation}
\label{eq:A-S_on_M_0}
\int_X I(u)\mathrm du=\int_{\pi X_0}\mathbb{E}_{x,x;t}\Big[ \lim_{t\downarrow 0^+}p_M(t,x,x)R_t\str \big(M_t \widehat X^{-1}_t\big)\Big]\mathrm dx
\end{equation}
(cf. the preceding footnote for this computation involving $\mathcal P_0^D$).

Hence, by \cite[Section 7.6]{hsu2002stochastic} or Remark \ref{rmk:computation of index} one readily obtains a proof of our index theorem by first finishing the calculation of \eqref{eq:A-S_on_M_0}  as $\widehat A(\mathcal H)\wedge \ch{\xi}$, mainly due to the fact that $x\in\pi X_0\subset M$ here is an ordinary point (i.e. the projection $\pi:X\to M$ gives a fiber bundle on a neighborhood of $X$). Here $\widehat A(\mathcal H)$ and $\ch\xi$ are respectively the equivariant $\widehat A$-genus of the subbundle $\mathcal H$ of $TX$ and the equivariant Chern character of the $\sone$-equivariant complex vector bundle $\xi$ (cf. \eqref{eq:Dirac_composition_new}). Hence they descend to $M$ at those $x\in\pi X_0$ so that we are capable of performing the needed calculation essentially in the classical manner at such $x\in M$. Recall that $\omega_0$ is given in Definition \ref{nota:omega0}. Then an easy lifting procedure by using $\omega_0$ in $\sone$-fiber integration gives 
(first back to $X_0$ then $X$) 
\begin{equation}
\label{eq:ind_component}
\ind_\sone (D_{X,0})=\int_X\left[\frac{1}{2\pi}\widehat A(\mathcal H)\wedge\ch\xi\wedge\omega_0\right]_{n+1(=2\ell+1)}
\end{equation}
where $[\cdots]_k$ denotes the $k$-form component. See Remark \ref{rmk:computation of index} for more.

Now, we have almost reached the following index theorem as promised in the Introduction.

\begin{thm}
\label{thm:main_result}
(A transversal, equivariant Atiyah-Singer index theorem) In the notation of Subsection \ref{subsection:transversal_D_X}, let $X$ be a compact, oriented manifold of dimension $2\ell +1$ with locally free $\,\sone$-action. Recall the horizontal subbundle $\,\mathcal H$ of $\,TX$. Assume that $X$ is transversally spin (equivalently, the second Stiefel-Whitney class $w_2(\mathcal H)$ of $\,\mathcal H$ is zero) and that for a choice of spin structure the $\,\sone$-action is of transversally even type. Then our transversal, equivariant index theorem is given by
\begin{equation}
\label{main}
\ind_\sone(D_{X,0})=\frac{p}{2\pi}\int_X \widehat A(\mathcal H)\wedge\ch\xi\wedge\omega_0
\end{equation}
where $\xi$ is an $\sone$-equivariant complex vector bundle and $\omega_0$ is as in Definition \ref{nota:omega0}. Here $p=\min_{u\in X}|H_u|$ ($H_u$ denotes the finite isotropy subgroup of $\;\sone$ at $u$).
\end{thm}
\begin{proof} Equip $\xi$ with an $\sone$-invariant connection.  
Now suppose that $\big\{ u\in X:H_u=\{\mathrm{id}\}\big\}=\emptyset$ so that $p>1$. Denote $\sigma:\sone\times X\to X$ the original $\sone$-action. Reset the $\sone$-action $\widetilde{\sigma}$ to be $\widetilde{\sigma}(e^{i\theta},x)\coloneqq \sigma(e^{i\theta/p},x)$, so $\big\{ u\in X: H_u=\{\mathrm{id}\}\big\}\neq\emptyset$ with respect to $\widetilde{\sigma}$. We are ready to finish the proof.  To use \eqref{eq:ind_component} via $\widetilde\sigma$, one should insert $p\omega_0$
to fit the normalization condition of our metric on $X$ (cf. Section \ref{sec:const}). This implies the equality \eqref{main}.
\end{proof}

\section{Interchangeability of Limits and Related Computations}
\label{sec: est_HBM}
We shall now show that the interchangeability assumption \eqref{assump:interchange} holds.


\subsection{Heat Kernel on Orbifolds}
To develop the bounds of $p_M(t,x,y)$ on orbifold (cf. \eqref{eq:ker_est_M_lower} and \eqref{eq:ker_est_M_upper}), which is well known when $M$ is a manifold, we make use of the asymptotic behavior of $p_M(t,x,y)$. These bounds will be to give the estimate for $\widehat X_t$ (cf. Proposition \ref{lemma:ut}), which is in turn used to show our interchangeability assumption \eqref{assump:interchange}.

For the construction of $p_M(t,x,y)$ on orbifolds, we briefly sketch some ingredients in order to fix the notation  
(\cite {dryden2008asymptotic} or \cite[Subsection 14.3]{duistermaat2013heat}).  According to \cite[Notation 3.8]{dryden2008asymptotic}, for a covering consisting of finitely many orbifold charts $(\widetilde W_\alpha, G_\alpha, \pi_\alpha)$, the parametrices $H^{(m)}_{\alpha}$ on $W_\alpha=\pi_\alpha(\widetilde W_\alpha)\subset M$ and the approximate solution $H^{(m)}$ are given in the following form (cf. \cite[(3.9)]{dryden2008asymptotic}):
\begin{equation}
\begin{split}
\label{eq:asymp_H}
&\forall m>n/2,\; (t,x,y)\in (0,\infty)\times W_\alpha\times W_\alpha\text{ and the natural distance }d_{\widetilde W_\alpha}(\cdot,\cdot)\text{ on }\widetilde W_\alpha\\
&\hspace{1cm}H^{(m)}_{\alpha}(t,x,y)=\sum_{\gamma\in G_\alpha}(2\pi t)^{-n/2}e^{-d_{\widetilde W_\alpha}(\widetilde x,\gamma \widetilde y)^2/2t}\Big(u_0(\widetilde x,\gamma \widetilde y)+\cdots+t^mu_m(\widetilde x,\gamma \widetilde y)\Big),\\
&\hspace{1cm}H^{(m)}(t,x,y)\coloneqq \sum_\alpha \psi_\alpha(x)\rho_\alpha(y) H_\alpha^{(m)}(t,x,y),
\end{split}
\end{equation}
for $C^\infty$ cut-off functions $\{\psi_\alpha\}_\alpha$ (being identically one in a small neighborhood $V_\alpha$ and supported in $W_\alpha$) and a partition of unity $\{\rho_\alpha\}_\alpha$ subordinate to $\{ W_\alpha\}_\alpha$.
\comment{
Let $q_{M,\alpha}(t,x,y)$ be the local heat kernel associated with $W_\alpha$ (and $H_\alpha^{(m)}(t,x,y)$) with
\begin{equation}
\label{eq:local_heat_kernel}
q_{M,\alpha}^{(m)}(t,x,y)=\sum_{\gamma\in G_\alpha}(2\pi t)^{-n/2}e^{-d_{\widetilde W_\alpha}(\widetilde x,\gamma\widetilde y)^2/2t}\big( u^M_0(\widetilde x,\gamma\widetilde y)+\cdots+t^mu^M_m(\widetilde x,\gamma\widetilde y)\big).
\end{equation}
The main difference between the paramatrix $H_\alpha^{(m)}(t,x,y)$ and $q_{M,\alpha}^{(m)}(t,x,y)$ is that the positivity of Riemannian invariants in \eqref{eq:local_heat_kernel} is guaranteed. For simplicity we abuse the notations $u_i^M$ by $u_i$ for $i=0,1,2,\cdots$. Define $q_M^{(m)}(t,x,y)\coloneqq \sum_\alpha \psi_\alpha(x)\rho_\alpha(y) q_{M,\alpha}^{(m)}(t,x,y)$.}

Following the successive approximation, let $R_m(t,x,y)=( {\partial}/{\partial t}-\Laplace_M/2 )H^{(m)}(t,x,y)$ and define the operator $A\sharp B$ with kernel $(A\sharp B)(t,x,y)\coloneqq \int_0^t \int_M A(t-s,x,z) B(s,z,y) \mathrm dz\mathrm ds$, which satisfies $((A\sharp B)\sharp C)=(A\sharp (B\sharp C))$, denoted by $A\sharp B\sharp C$. 

Then the fundamental solution $p_M(t,x,y)$ of the heat equation on $M$ is given by
\begin{equation}
\label{eq:asymp_expand}
\begin{split}
p_M(t,x,y)&=H^{(m)}(t,x,y)-(H^{(m)}\sharp R_m)(t,x,y)+\cdots+(-1)^{j}(H^{(m)} \sharp R_m^{\sharp j})(t,x,y)+\cdots\\
&=H^{(m)}(t,x,y)+H^{(m)}(t,x,y)\sharp Q_m
\end{split}
\end{equation}
where $R_m^{\sharp k}:={ R_m\sharp\cdots \sharp R_m}$ ($k$ copies, $k\ge 2$) with $R_m^{\sharp 1}=R_m$ and the alternating sum $Q_m$ (of $R_m^{\sharp j}$) satisfies for $m>{n}/{2}+2$, $|Q_m(t,x,y)|\le Ct^{m-{n}/{2}}\text{ on }[0,T]\times M\times M$ (for some $C>0$) (cf. \cite[Lemma 3.18, p. 217]{dryden2008asymptotic}).

As $t\downarrow 0^+$ and $m\to\infty$, $H^{(m)}(t,x,y)$ can approximate $p_M(t,x,y)$ to an arbitrarily high degree (in $t$). It follows that the asymptotic expansion of the heat kernel is  locally of the form:
\begin{equation}
\label{eq:asymp_pm}
p_M(t,x,y)\sim \sum_{\gamma\in G_\alpha}(2\pi t)^{-n/2}e^{-d_{\widetilde W_\alpha}(\widetilde x,\gamma \widetilde y)^2/2t}\big(u_0(\widetilde x,\gamma \widetilde y)+tu_1(\widetilde x,\gamma \widetilde y)+\cdots\big).
\end{equation}
See \eqref{eq:est_convol} for an improvement of these estimates essentially by $e^{-d(x,y)^2/2t}$.

\begin{rmk}
\label{rmk:local_heat_kernel}
The above idea can be adapted to the following, which our subsequent discussion is based. In fact, the same heat kernel $p_M(t,x,y)$ can be derived using local heat kernels $q_{\alpha}(t,x,y)$ (cf. \cite[p. 177]{duistermaat2013heat}). 
First extend $\widetilde W_\alpha$ (with its metric) to a compact manifold (e.g. $\mathbb{S}^n$ with a compatible Riemannian metric). Let $\widetilde q_{\alpha}(t,\widetilde x,\widetilde y)$ be the heat kernel on this compact manifold and consider $\widetilde q_{\alpha}(t,\widetilde x,\widetilde y)|_{\widetilde W_\alpha}$, still denoted by $\widetilde q_{\alpha}(t,\widetilde x,\widetilde y)$. As similar to the above discussion, the $G_\alpha$-invariant function $\sum_{\gamma\in G_\alpha} \widetilde q_{\alpha}(t,\widetilde x,\gamma\widetilde y)$ descends to a well-defined function $q_{\alpha}(t,x,y)$ on $W_\alpha$, which is going to be used for the successive approximation as above. By the maximum principle $\widetilde q_\alpha$ hence $q_\alpha$ is strictly positive on $(0,\infty)\times W_\alpha\times W_\alpha$ (cf. \cite[Theorem 1, p. 181]{chavel1984eigenvalues}). This procedure yields (cf. \cite[p. 154 (45)]{chavel1984eigenvalues}), for any fixed $m>{n/2+2}$
\begin{equation}
\label{eq:local_heat_kernel}
(\sum_\gamma\widetilde q_\alpha=)q_{\alpha}(t,x,y)=\sum_{\gamma\in G_\alpha}\frac{e^{-d_{\widetilde W_\alpha}(\widetilde x,\gamma\widetilde y)^2/2t}}{(2\pi t)^{n/2}}\Big( \eta(d_{\widetilde W_\alpha}(\widetilde x,\gamma\widetilde y))\sum_{j=0}^m t^ju_j(\widetilde x,\gamma\widetilde y)+O(t^{m+1})\Big),
\end{equation}
where $\eta(\cdot)$ is a function such that $\eta|[0,\epsilon/4]=1$ and $\eta|[\epsilon/2,\infty]=0$ for some $\epsilon>0$ (cf. \cite[p. 151]{chavel1984eigenvalues}). Likewise setting $q(t,x,y)=\sum_{\alpha}\psi_\alpha(x)\rho_\alpha(y)q_\alpha(t,x,y)$ as in \eqref{eq:asymp_H}, we have the corresponding $\mathscr R_m$ and $\mathscr Q_m$ as in \eqref{eq:asymp_expand}.
\end{rmk}

We are now ready to establish our bounds for $p_M(t,x,y)$. The following two propositions are the main result of this subsection. For basics about the distance function $d(\cdot,\cdot)$ on orbifolds (e.g. its existence), see \cite[Chapter III, p. 586]{bridson1999metric} or \cite{lange2020orbifolds} and references therein.

\begin{prop}
\label{prop:ker_est_M_lower}
There exists $\epsilon>0$ such that
\begin{equation}
\label{eq:ker_est_M_lower}
p_M(t,x,y) \ge \frac{C_1}{t^{n/2}}e^{- d(x, y)^2/2t}
\end{equation}
for any $t\in (0,\epsilon)$ and $y\in B(x;r)$ with $r<\epsilon$ where $C_1$ denotes some universal constant independent of the choice of $x$ and $y$, and $d(\cdot, \cdot)$ denotes the distance function on $M$.
\end{prop}
\begin{proof}
Here the arguments are different from \cite[Theorem 5.3.4]{hsu2002stochastic} not only because our $M$ is an orbifold but also because we only need the small time behavior of $p_M(t,x,y)$ for near points $x$, $y$ instead of all $(x,y)\in M\times M$ (as in \cite{hsu2002stochastic}). In fact, we find that arguments in \cite[p. 151--155]{chavel1984eigenvalues} are more suitable to our needs.

A word of caution is in order. Since the error term arising from \eqref{eq:asymp_pm} is of order $O(t^N)$ (for any fixed integer $N>0$), this error is still too large to validate \eqref{eq:ker_est_M_lower} due to the fast decay of the nontrivial exponential factor in \eqref{eq:asymp_pm} for $x\neq y$. The \eqref{eq:asymp_pm} does not seem directly applicable here.

One advantage of local heat kernels is the {\it positivity}: $\widetilde q_\alpha(t,\widetilde x,\gamma\widetilde y)>0$ (see Remark \ref{rmk:local_heat_kernel}).  Using it for a prescribed pair of near points $x$ and $y$ with $t\ll 1$ yields by \eqref{eq:local_heat_kernel}
\begin{equation}
\label{eq:key_obser_1}
q_{\alpha}(t,x,y)\ge \frac{e^{-d(x,y)^2/2t}}{(2\pi t)^{n/2}}\Big( \eta(d(x,y))\sum_{j=0}^m t^ju_j(\widetilde x,\gamma_0\widetilde y)+O(t^{m+1})\Big)\;\text{ for some }\gamma_0.
\end{equation}
It is derived by using the positivity above to only keep the term with $\gamma_0\in G_\alpha$ which satisfies $d_{\widetilde W_\alpha}(\widetilde x,\gamma_0\widetilde y)=\min_{\gamma\in G_\alpha}d_{\widetilde W_\alpha}(\widetilde x,\gamma\widetilde y)$ and by the fact that $\min_{\gamma\in G_\alpha}d_{\widetilde W_\alpha}(\widetilde x,\gamma\widetilde y)$ equals $d(x,y)$ for appropriate choice of orbifold charts $\widetilde W_\alpha$ containing $x, y$ (cf. \cite[Definition 1.1]{lange2020orbifolds} together with a version of Hopf-Rinow Theorem on $M$ \cite[p. 35]{bridson1999metric}). See also Lemma \ref{:near_plemmats} below. The requirement that $x$ and $y$ are near points ensures $\eta(d(x,y))>0$.

Using \cite[Lemmas 1--3, p. 152--153]{chavel1984eigenvalues} for our $q(t,x,y)$ (see Remark \ref{rmk:local_heat_kernel} above) and after going through similar computations using $d(x,y)$ in place of $d_{\widetilde W_\alpha}(\widetilde x,\gamma\widetilde y)$ for \cite[top lines in p. 154]{chavel1984eigenvalues} in a way similar to \eqref{eq:key_obser_2} below, we can obtain the following estimation (cf. \cite[p. 154 and (39), p. 152]{chavel1984eigenvalues}):
\begin{equation}
\label{eq:est_convol}
|\mathscr R_m^{\sharp j}(t,x,y)|\le \frac{A^jV^{j-1}t^{m+j-n/2+1}e^{-d(x,y)^2/2t}}{(m+j-n/2-1)\cdots(m+1-n/2)}
\end{equation}
where $A=\sup_{[0,1]\times M\times M}\big|\frac{\mathscr R_m}{t^{m-n/2}e^{-d(x,y)^2/2t}}\big|$ and $V=\mathrm{vol}(M)$.
Combining \eqref{eq:asymp_expand} (for $q(t,x,y)$ in place of $H^{(m)}$ by Remark \ref{rmk:local_heat_kernel}), \eqref{eq:est_convol} (including a similar estimate on $q\sharp \mathscr Q_m$ by using \eqref{eq:asymp_expand} 
for $\mathscr Q_m$)
and \eqref{eq:key_obser_1}, leads to (cf. \cite[p. 154]{chavel1984eigenvalues})
\begin{equation}
\label{eq:heat_kernel_ge}
p_M(t,x,y)\ge \frac{e^{-d(x,y)^2/{2t}}}{(2\pi t)^{n/2}}\Big\{\psi(x,y)\sum_{j=0}^m t^ju_j(\widetilde x,\gamma_0\widetilde y)+O(t^{m+1})\Big\}\text{ for }m>n/2+2
\end{equation}
where $u_0(\widetilde x,\widetilde x)=1$ and $\psi$ a cut-off function ($\equiv 1$ near $x=y$). 

There exist $\epsilon_1>0$ such that for $\gamma_0\widetilde y\in B(\widetilde x;\epsilon_1)$, $|u_0(\widetilde x,\gamma_0\widetilde y)-u_0(\widetilde x,\widetilde x)|\le 1/2$ (which implies $u_0(\widetilde x,\gamma_0\widetilde y)\ge 1/2>0$)  and $\epsilon_2>0$ such that $\left|\sum_{j= 1}^mt^ju_j(\widetilde x,\gamma_0\widetilde y)\right|\le u_0(\widetilde x,\gamma_0\widetilde y)/2$ for $t\in (0,\epsilon_2)$. Also, there exists $\epsilon_3>0$ such that $\psi(x,y)=1$ for $y\in B(x;\epsilon_3)$. Choosing $\epsilon=\min\{\epsilon_1,\epsilon_2,\epsilon_3\}$ and making it 
independent of the choice of $x$ by compactness of $M$, we obtain the proposition from \eqref{eq:heat_kernel_ge}.
\end{proof}

\begin{rmk}
\label{distant}
For use in the semimartingale property (see Ramark \ref{rmk:Brownian_bridge}) 
it is desirable to extend the above estimate to 
distant points (cf. \cite[p.141]{hsu2002stochastic} on manifolds).   For orbifolds we refer to 
Appendix  \ref{appendix:distant} for a proof. 
\end{rmk}

\begin{prop}
\label{prop:ker_est_M_upper}
For every $t\in (0,1)$ and $x$, $y\in M$ (not necessarily near points),
\begin{equation}
\label{eq:ker_est_M_upper}
p_M(t,x,y) \le \frac{C_2}{t^{n/2}}e^{- d(x, y)^2/2t}\,\,\, (\le  \frac{C_2}{t^{n/2}}) 
\end{equation}
where $C_2$ denotes some universal constant independent of $t$ and the choice of $x$ and $y$.
\end{prop}

\begin{proof}

For any $t\in (0,1)$ and any near points $x$, $y\in M$ (before the successive approximation),
\begin{equation}
\label{eq:key_obser_2}
\begin{split}
\big(q_{\alpha}(t,x,y)=\big)&\sum_{\gamma\in G_\alpha}\frac{e^{-d_{\widetilde W_\alpha}(\widetilde x,\gamma\widetilde y)^2/2t}}{(2\pi t)^{n/2}}\Big( \eta(d_{\widetilde W_\alpha}(\widetilde x,\gamma\widetilde y))\sum_{j=0}^m t^ju_j(\widetilde x,\gamma\widetilde y)+O(t^{m+1})\Big)\\
&\hspace{1cm}\le \frac{e^{-d(x, y)^2/2t}}{(2\pi t)^{n/2}}\sum_{\gamma\in G_\alpha}\Big( \eta(d_{\widetilde W_\alpha}(\widetilde x,\gamma\widetilde y))\sum_{j=0}^m t^ju_j(\widetilde x,\gamma\widetilde y)+O(t^{m+1})\Big).
\end{split}
\end{equation}
The inequality merely replaces each $d_{\widetilde W_\alpha}(\widetilde x,\gamma\widetilde y)$ in the exponential term by $d_{\widetilde W_\alpha}(\widetilde x,\gamma_0\widetilde y)$ (cf. \eqref{eq:key_obser_1}). Following arguments similar to that from \eqref{eq:est_convol} to \eqref{eq:heat_kernel_ge}, after the successive approximation using especially \cite[Lemma 3, p. 153]{chavel1984eigenvalues} one sees \eqref{eq:ker_est_M_upper} for $x$, $y$ which are not necessarily near points.
\end{proof}

The following remark will be needed in \eqref{eq:boundedness} and \eqref{rmk:degen_case}.
\begin{rmk}
\label{rmk:x_0}
At a given point $x_0\in M$, we claim $p_M(t,x_0,x_0)$ $\sim \left({1}/{2\pi t}\right)^{n/2}\cdot p\;\text{as $t\to 0^+$}$ where $p$ is the order of the isotropy subgroup (of $\sone$) at $x_0$. For, by arguments similar to the proof of Proposition \ref{prop:ker_est_M_lower} (see also Lemma \ref{:near_plemmats} below), we have $p_M(t,x_0,x_0)\ge \left({1}/{2\pi t}\right)^{n/2}\cdot p$ (as $t$ small). A similar argument with \eqref{eq:ker_est_M_upper} gives the upper bound $\left({1}/{2\pi t}\right)^{n/2}\cdot p (1+O(t))$. 
\end{rmk}

\subsection{Distance on Orbifolds}
We first review Hsu's approach that gives a gradient estimate for the heat kernel. Then we show that the similar gradient estimate still holds for orbifolds.


This gradient estimate (cf. Theorem \ref{thm:grad_est}) is implied by a lemma for lifted functions on $\ortho{M}$ (cf. Lemma \ref{lemma:ineq_lift_func}) and by the bounds for the associated heat kernel (cf. Propositions \ref{prop:ker_est_M_lower} and \ref{prop:ker_est_M_upper}). Let us first set up the notation\footnote{In what follows the Laplacian, distance function, etc. work as if they were under the (smooth) manifold setting by the discussion in Appendix \ref{appendix:adaptation}. See also Subsection \ref{subsection: orbibundle}. For the foundation of probabilistic aspects (e.g. Brownian motion) on orbifolds, see Appendix \ref{appendix:L_diffusion}.}.

\begin{defn}
Define $J(t,\widehat x)\coloneqq \ln p_M(T-t,\widehat\pi\widehat x,y)$ as the lift of $\ln  p_M(t,x,y)$ on $\ortho{M}$. The horizontal gradient is $\nabla^HJ=\big( H_1J,\dotsc, H_nJ \big)$.
\end{defn}

\begin{lemma}
\label{lemma:ineq_lift_func}
Recall that $M=X/\sone$ is a compact orbifold and let 
\[
E_i=\frac{1}{2}\left[\Laplace_\ortho{M}^H,H_i\right]J+\left\langle\nabla^H H_iJ,\nabla^H J\right\rangle-\frac{1}{2}H_i\left\langle\nabla^H J,\nabla^H J\right\rangle.
\]
There is a constant $C$ such that $|E_i|\le C|\nabla^H J|$.
\end{lemma}
\begin{proof}
This lemma is part of \cite[Lemma 5.5.2]{hsu2002stochastic}. We note that for an ordinary manifold $N$, the proof of \cite{hsu2002stochastic} is given for the orthonormal frame bundle $\mathscr{O}(N)$. Nevertheless, these arguments of \cite{hsu2002stochastic} remain applicable to our situation (after going to orbifold charts on $M$).
\end{proof}

\begin{rmk}
The second order estimate $E_{ij}$ in \cite{hsu2002stochastic} is not used in the present work.
\end{rmk}

The needed gradient estimate, for short time and near points,  goes as follows.

\begin{thm}
\label{thm:grad_est}
Let $M$ be a compact orbifold. Then there are constants $\,C>0$, $\delta>0$ such that for any $T\in (0,\delta)$ and near points $x$, $y\in M$,
\[
|\nabla \ln p_M(T,x,y)|\le C \left[ \frac{d(x,y)}{T}+\frac{1}{\sqrt{T}}\right].
\]
\end{thm}

\begin{proof}
We divide the proof into three steps. The first step is to reach
\begin{equation}
\label{eq:ker_1}
\mathbb{E}\int_0^{T/2} |\nabla^H J(s,\widehat X_s)|^2\mathrm ds\le C_1\left[\frac{d(x,y)^2}{T}+1\right]
\end{equation}
as in \cite[(5.5.9)]{hsu2002stochastic}. Here essential use is made of the estimates in our Propositions \ref{prop:ker_est_M_lower} and \ref{prop:ker_est_M_upper} on $p_M(t,x,y)$. The second step uses Lemma \ref{lemma:ineq_lift_func} to readily get $T|\nabla^H J(0,\widehat x_0)|\le C_3\mathbb{E}\int_0^{T/2}|\nabla^HJ(s,\widehat X_s)|\mathrm ds$, and then applies the Cauchy-Schwarz inequality to the RHS. This, together with \eqref{eq:ker_1}, concludes the proof of Theorem \ref{thm:grad_est}. For further details we refer to \cite[Theorem 5.5.3]{hsu2002stochastic} (in which the proof uses \cite[Corollary 5.3.5]{hsu2002stochastic} and is not restricted to short time and near points).
\end{proof}

\begin{rmk}
\label{gradientlogheatkernel}
To extend the result to distant points (due to the stochastic process, 
cf. Remark \ref{rmk:Brownian_bridge} and Appendix \ref{appendix:distant}), the lower bound estimate 
for distant points (cf. Proposition \ref{prop:ker_est_M_lower} and Remark \ref{distant})  
will be useful (cf. \cite[(5.5.8) and the equation below it]{hsu2002stochastic});  see Appendix \ref{appendix:distant}. 
\end{rmk} 

The gradient estimate just derived enables one to prove the following inequality.

\begin{thm}
\label{thm:distance_est}
There are constants $\,C>0$, $\delta>0$ such that for every $T\in (0,\delta)$ and near points $x$ and $y\in M$, $\mathbb{E}_{x,y;t}d(X_t,y)^2\le C\left( d(x,y)^2+\min\{t,T-t\}\right),\,\, 0\le t\le T$.
\end{thm}
\begin{proof}
We shall adapt Hsu's approach to our situation (cf. \cite[Proposition 5.5.4]{hsu2002stochastic}).  There, a smooth function $f(z)=|z|^2$ is considered in a neighborhood $O$ of $y$ with normal coordinates $z=\{z^i\}$ and 
satisfies:
\begin{equation}
C^{-1}d(y,z)^2\le f(z)\le C d(y,z)^2,\quad |\nabla f(z)|\le \widetilde C\sqrt{f(z)}. \label{eq:dist_func}
\end{equation}
In fact $\sqrt{f(z)}=|z|$ is exactly the distance function $d(y,z)$ since $d(x_0,\expo_{x_0}{x})=|x|$
in some normal coordinates (cf. \cite[Proposition 1.27]{berline2003heat}).

In our case where $M$ is an orbifold, we are motivated to define $f(z)$ as the distance function 
(cf. \eqref{eq:key_obser_1})
\[
f(z)=d(y,z)^2:=\min_{\gamma\in G_\alpha}d_{\widetilde W_\alpha}(\widetilde y,\gamma \widetilde z).
\]

To finish the proof, we have the following lemma for $f(z)$:
\begin{lemma}
\label{:near_plemmats}
Given two near points $\widetilde y$, $\widetilde z\in\widetilde W_\alpha$ where $(\widetilde W_\alpha,G_\alpha,\pi_\alpha)$ denoting an orbifold chart of M is so small\footnote{Given any $p_\alpha\in X$ and $\widetilde W_\alpha$ a $G_\alpha$-invariant slice through $p_\alpha$ with $G_\alpha$ the finite isotropy subgroup at $p_\alpha$ (see Appendix \ref{appendix:adaptation} or \cite[p. 173]{duistermaat2013heat}), suppose $h_ix_i=x_i\in\widetilde W_\alpha$ for some sequence $x_i\to p_\alpha$, with $h_i\in\sone\setminus G_\alpha$. Then $\overline{\{h_i\}}\cap G_\alpha$ is non-empty. This is absurd since the union of (finite) isotropy subgroups (in $\sone$) over all $x\in X$ is a finite set by the finiteness of the number of strata.} that at any $\widetilde x\in\widetilde W_\alpha$ the isotropy subgroup $G_{\widetilde x}$ is contained in $G_\alpha$.
\begin{enumerate}
\item[(i)] If the isotropy subgroup $G_0$ of $G_\alpha$ at $\widetilde y$ is not the whole $G_\alpha\,$, then
\[
\min_{\gamma\in G_\alpha}d_{\widetilde W_\alpha}(\widetilde y,\gamma \widetilde z)=d_{\widetilde W_\alpha}(\widetilde y, \widetilde z)<\min_{\gamma\in G_\alpha\setminus G_0} d_{\widetilde W_\alpha}(\widetilde y,\gamma\widetilde z).
\]
\item[(ii)] If $G_0=G_\alpha\,$ then for every $\gamma\in G_\alpha$, $d_{\widetilde W_\alpha}(\widetilde y,\gamma\widetilde z)=d_{\widetilde W_\alpha}(\widetilde y,\widetilde z)$.
\end{enumerate}

\end{lemma}
\noindent\textit{Proof of Theorem \ref{thm:distance_est} continued.} With the above lemma using the orbifold chart $\widetilde W_\alpha$, the smooth function $f(z)=|z|^2$ still serves as the distance function for near points, and \eqref{eq:dist_func} in our orbifold case holds true as well. With \eqref{eq:dist_func} established, the remaining arguments of Hsu go through without change (nonetheless, Hsu's original formulation is not restricted to short time and near points).
\end{proof}

\noindent\textit{Proof of Lemma \ref{:near_plemmats}.} Let us fix $\widetilde y\in\widetilde W_\alpha$. For $\gamma\in G_0(\subset G_\alpha)$ the isotropy subgroup at $\widetilde y$, one sees
\begin{equation}
\label{eq:distance_on_W_tilde}
d_{\widetilde W_\alpha}(\widetilde y,\widetilde z)=d_{\widetilde W_\alpha}(\gamma \widetilde y,\gamma \widetilde z)=d_{\widetilde W_\alpha}(\widetilde y,\gamma\widetilde z).
\end{equation}
The assertion (ii) follows immediately if $G_0=G_\alpha$.

Next assume $G_0\neq G_\alpha$ of assertion (i). Write $ G_0g_1,\, G_0g_2,\,\dotsc $ as (finitely many) cosets of $G_\alpha$ where $g_1=\mathrm{id}$. It follows, as in \eqref{eq:distance_on_W_tilde}, that for every $\gamma_i\in G_0g_i$ $(i=1,2,\cdots)$ $d_{\widetilde W_\alpha}(\widetilde y,\gamma_i \widetilde z)=d_{\widetilde W_\alpha}(\widetilde y,g_i\widetilde z)$,
hence that
\begin{equation}
\label{eq:distance_min_change}
\min_{\gamma\in G_\alpha}d_{\widetilde W_\alpha}(\widetilde y,\gamma\widetilde z)=\min_{g_1,\,g_2,\,\cdots}d_{\widetilde W_\alpha}(\widetilde y,g_i \widetilde z).
\end{equation}

Let $O_i=B_{\delta_i}(g_i\widetilde y)$ be the open ball centered at $g_i\widetilde y$ with radius $\delta_i\ll 1$ and $\overline{O}_i$ the closure of $O_i$ such that these $\overline{O}=\overline{O}_1$, $\overline{O}_2,\cdots$ are pairwise disjoint. Set $\varepsilon_i=\inf_{\widetilde z\in O,\,\widetilde z_i\in O_i}d_{\widetilde W_\alpha}(\widetilde z,\widetilde z_i)> 0\text{ for }i>1$. Given points $\widetilde y$, $\widetilde z$ that are close enough (and thus $\widetilde z\in O$), since $g_i\widetilde z\in O_i$ and $\overline O_i$ is disjoint from $\overline O$ (for $i>1$) we have  $\delta_1+\varepsilon_i\le  d_{\widetilde W_\alpha}(\widetilde y,g_i\widetilde z)\text{ for }i>1$. This and \eqref{eq:distance_min_change} yield $\displaystyle\min_{\gamma\in G_\alpha}d_{\widetilde W_\alpha}(\widetilde y,\gamma \widetilde z)=d_{\widetilde W_\alpha}(\widetilde y, \gamma\widetilde z)\big|_{\gamma=\mathrm{id}}\le \delta_1$ and $d_{\widetilde W_\alpha}(\widetilde y,\widetilde z)<\displaystyle\min_{\gamma\in G_\alpha\setminus G_0}d_{\widetilde W_\alpha}(\widetilde y,\gamma\widetilde z)$ as claimed.
\qed

\begin{rmk}
\label{rmk:Brownian_bridge}
For the reasoning throughout (the orbifold case), the needed semimartingale property (cf. \eqref{eq:F-K_X}) of the horizontal lift (to $X$) of a Brownian bridge can be proved by following a similar argument of \cite[Proposition 5.5.6]{hsu2002stochastic}. The original proof is mainly based on \cite[Lemma 5.4.2]{hsu2002stochastic} (and 
$|\nabla \ln p_M(T-s, X_s, y)|$ in \cite[(5.5.11)]{hsu2002stochastic}). For our present situation (Brownian bridges for near points and short time) we can use Propositions \ref{prop:ker_est_M_lower} and \ref{prop:ker_est_M_upper} (in analogy with Hsu's Corollary 5.3.5 for his proof of Lemma 5.4.2).  For handling the gradient of logarithmic heat kernels, see Remark \ref{gradientlogheatkernel} and Appendix \ref{appendix:distant}.  Hence the result \cite[Proposition 5.5.6]{hsu2002stochastic} is still applicable to our orbifold $M$ (and $X$).
\end{rmk}

\subsection{Estimation of Horizontal Brownian Motion}
\label{subsection:Estimation of Horizontal Brownian Motion}
To justify \eqref{eq:A-S_on_M_0}, the estimation of horizontal Brownian motion and a property of supertrace are needed in the process. For the former, the main estimate is in Proposition \ref{lemma:ut} below.

\begin{rmk}
\label{rmk:111}
(cf. \cite[p. 220]{hsu2002stochastic}) (i) We write the stochastic parallel transport $\widehat X_t^{-1}$ of \eqref{eq:before_exchange} on $G=\mathscr{S}(M)\otimes \xi$ as:
$\widehat X_t^{-1}=(\widehat{X}_t^{-1})^{\mathscr S(M)}\otimes (\widehat X_t^{-1})^{\xi}$ on the respective bundles.  
(ii) Practically $(\widehat{X}_t^{-1})^{\mathscr S(M)}$ as an action on $\mathscr S(M)_x$ is close to the identity so there exists a $v_t^{\mathscr S(M)}\in\mathfrak{so}(n)$ such that $(\widehat X_t^{-1})^{\mathscr S(M)}=\expo (v_t^{\mathscr S(M)})$ in $\mathrm{Spin}(n)$.
\end{rmk}
\begin{prop}
\label{lemma:ut}
In the notation of Remark \ref{rmk:111}, for any positive integer $N$, there is a constant $K_N$ such that $\mathbb{E}_{x,x;t}|\widehat X_t^{\mathscr S(M)}-I^{\mathscr S(M)}|^N\le K_Nt^N$ for small $t$. 
\end{prop}

\begin{proof}
We only need the small $t$ behavior of the horizontal Brownian motion $\widehat X_t^{\mathscr S(M)}$.  By Remark \ref{rmk:Brownian_bridge}, Theorems \ref{thm:grad_est} and \ref{thm:distance_est} of the last subsection, the argument of \cite[Lemma 7.3.4]{hsu2002stochastic} can be applied to our orbifold $M$. Another key is the SDE for $\widehat X_t^{\mathscr S(M)}$ (cf. \cite[Section 7.7 and Theorem 5.4.4]{hsu2002stochastic}). Such SDE is basically derived from Girsanov's theorem, which is fortunately also applicable to our situation (cf. Remark \ref{rmk:Brownian_bridge} above). The remaining proof follows from local computations (cf. \cite[Lemma 7.3.4]{hsu2002stochastic}).
\end{proof}

We turn now to the second property -- supertrace as mentioned above.

\begin{lemma} (cf. \cite[Lemma 7.4.3]{hsu2002stochastic})
\label{cor: str_est}
Let $A_1,\cdots,A_\ell\in\mathfrak{so}(n=2\ell)$. If $k<\ell$, then $\str (D^\ast A_1\circ \cdots \circ D^\ast A_k)=0$, where the action $D^\ast A:\Laplace^\pm\to\Laplace^\pm$ (cf. Subsection \ref{subsection:Notations and Set Up}) of $A=(a_{ij})\in\mathfrak{so}(n)$ on $C(\mathbb{R}^n)$ is given by the Clifford multiplication on the left by ${1}/{4}\sum_{1\le i,j\le n}a_{ij}e_i e_j$.
\end{lemma}
\begin{rmk}
\label{rmk:exterior_algebra}
The proof of the above lemma employs the following identity which is needed again in the next subsection (cf. \eqref{eq:Ut_expansion}) $\exp tA(x)=\sum_{k=0}^\infty ({t^k}/{k!})D^\ast A(x)^k$.
\end{rmk}
\begin{rmk}
\label{cor:Clifford_prop}
The Clifford multiplication $c(X_i)c(X_j)$ (in \eqref{eq:Lichnerowicz_lifted}) can be identified as $2\,\widehat x\circ D^\ast A_{ij}\circ \widehat x^{-1}$ for some $A_{ij}\in\mathfrak{so}(n)$.  See \eqref{eq:str_zero} for use.
\end{rmk}

\subsection{Interchangeability in \eqref{assump:interchange}}
Despite that the idea of our arguments parallels to some extent that of Hsu \cite[Theorems 7.3.5 and 7.6.2]{hsu2002stochastic}, we must verify the interchangeability of our own. For that reason and the sake of clarity, we choose to write down the details
and adapt them to orbifolds  (cf. Appendix \ref{appendix:L_diffusion}).  

For the interchangeability one may ask for the uniform boundedness of the integrand in the RHS of \eqref{eq:before_exchange}.  Recall the notation $\widehat{X}_t^{-1}$, $(\widehat{X}_t^{-1})^{\mathscr{S}(M)}$ 
as in Remark \ref{rmk:111}.
We continue with \eqref{eq:A-S_on_M_0}, and set 
\begin{equation}
\begin{split}
\label{eq:index_on_M}
I(t,x)&=\mathbb{E}_{x,x;t}\left[ R_t\str\left( M_t\widehat X^{-1}_t\right)\right]{p_M(t,x,x)}\\
\comment{&=\frac{1}{2\pi}\mathbb{E}_{x,x;t}\left[ R_t\str\left( M_t\widehat X^{-1}_t\right)\right]p_M(t,x,x)\\
&=\frac{1}{2\pi}\mathbb{E}_{x,x;t}\left[ R_t\str\left( M_t\big( {(\widehat{X}_t^{-1})}^{\mathscr{S}(M)} \otimes {(\widehat{X}_t^{-1})}^\xi\big)\right)\right]p_M(t,x,x)}
\end{split}
\end{equation}
for $x\in M$.   One sees that it is enough to show the uniform boundedness of $I(t, x)$.  

One first expands $M_t$ in \eqref{eq:index_on_M} into series (cf. Subsection \ref{subsection:Prob}): $M_t=\sum_{i=0}^\ell m_i(t)+Q(t)$, where $m_i(t)= -\frac{1}{4}\int_0^t m_{i-1}(s)\sum_{j,k=1}^n c(e_j)c(e_k)\otimes\widehat X_s L(\widehat X_s e_j,\widehat X_s e_k)\widehat X_s^{-1}\mathrm ds,\; m_0=I$ and $|Q(t)|\le C_1 t^{\ell+1}$.

Next, to expand $(\widehat X_t^{-1})^{\mathscr S(M)}$ (in \eqref{eq:index_on_M}) into series, one has (see Remark \ref{rmk:exterior_algebra}),
\begin{equation}
\label{eq:Ut_expansion}
(\widehat X_t^{-1})^{\mathscr S(M)}=\expo(v_t^{\mathscr S(M)}) =\sum_{k=0}^\ell \frac{\big(D^\ast v_t^{\mathscr S(M)}\big)^k}{k!} +R(t),\quad 0<t\ll 1,
\end{equation}
for some $v_t^{\mathscr S(M)}\in\mathfrak{so}(n)$ such that $(\widehat X_t^{-1})^{\mathscr S(M)}=\expo(v_t^{\mathscr S(M)})$ in ${\mathrm{Spin}}(n)$. By using Proposition \ref{lemma:ut} and Remark \ref{rmk:exterior_algebra} the remainder $R(t)$ satisfies $\mathbb{E}_{x,x;t}|R(t)|\le C_2 t^{\ell+1}$.

Now, $M_t (\widehat X_t^{-1})^{\mathscr S(M)}$ in \eqref{eq:index_on_M} can be computed via \eqref{eq:Ut_expansion}:
\begin{equation}
\label{eq:MtUt}
M_t(\widehat X_t^{-1})^{\mathscr S(M)}=\sum_{i,j\le \ell}\frac{m_i(t)(D^\ast v_t^{\mathscr S(M)})^j}{j!} +S(t),
\end{equation}
where the remainder $S(t)$ satisfies $\mathbb{E}_{x,x;t} |S(t)|\le C_3t^{\ell +1}$. By Remark \ref{cor:Clifford_prop} and Lemma \ref{cor: str_est}, one sees
\begin{equation}
\label{eq:str_zero}
\str\big( m_i(t)(D^\ast v_t^{\mathscr S(M)})^j\big)=0\text{ if } i+j<\ell.
\end{equation}
This immediately leads \eqref{eq:MtUt} to 
\begin{equation}
\label{eq:MtUt_str}
\str\big(M_t(\widehat X_t^{-1})^{\mathscr S(M)}\big)=\str\Big( \sum_{i+j\ge \ell}\frac{m_i(t)(D^\ast v_t^{\mathscr S(M)})^j}{j!} +S(t)\Big).
\end{equation}

By $\mathbb{E}_{x,x;t}|m_i(t)(D^\ast v_t^{\mathscr S(M)})^j|\le C_4 t^{i+j}$, $\mathbb{E}_{x,x;t}|S(t)|\le C_3t^{\ell+1}$ and applying Propositions \ref{prop:ker_est_M_lower} and \ref{prop:ker_est_M_upper} (with $n/2=\ell$ here), one sees that as $t\downarrow 0^+$,
\begin{equation}
\label{eq:boundedness}
\mathbb{E}_{x,x;t}\big(\str \big(M_t (\widehat X_t^{-1})^{\mathscr S(M)}\big)\big)p_M(t,x,x) \text{ is locally uniformly bounded in }x,
\end{equation}
leading to the interchangeability as asserted in this subsection, cf. Remark \ref{rmk_explain}. 

For the precise evaluation of \eqref{eq:index_on_M} we remark the following.

\begin{rmk}
\label{rmk:computation of index}
By using $\widehat{X}_t^{-1}={(\widehat{X}_t^{-1})}^{\mathscr{S}(M)} \otimes {(\widehat{X}_t^{-1})}^\xi$ (Remark \ref{rmk:111}), the computation of \eqref{eq:index_on_M} now boils down to
\begin{align}
&\lim_{t\downarrow 0^+}\mathbb{E}_{x,x;t}\bigg[R_t\str\big(M_t\widehat X^{-1}_t\big)\bigg]p_M(t,x,x)\nonumber\\
=&\lim_{t\downarrow 0^+}\mathbb{E}_{x,x;t}\bigg[R_t\str\Big(\sum_{i+j= \ell}\frac{m_i(t)\big((D^\ast v_t^{\mathscr S(M)})^j\otimes (\widehat X_t^{-1})^{\xi}\big)}{j!}\Big)\bigg]p_M(t,x,x).
\label{rmk:degen_case}
\end{align}

For those $x$ not in the singular locus of the orbifold $M$, \eqref{rmk:degen_case} is exactly bringing us back to $I(x)$ in \cite[the fifth line in p. 221]{hsu2002stochastic} (since $p_M(t,x,x)\sim\left({1}/{2\pi t}\right)^{n/2}$ by Remark \ref{rmk:x_0} for $p=1$ here).  The remaining arguments are essentially local in nature (see also \cite[Subsection 7.7, 
p. 222]{hsu2002stochastic}) and are omitted.
\end{rmk}
\begin{rmk}
\label{rmk_explain}
For \eqref{eq:boundedness} we remark that in the entire Section \ref{sec: est_HBM} including this subsection, we are not restricting ourselves to the nonsingular part (i.e. the principal stratum) of the orbifold $M$, until Remark \ref{rmk:computation of index} (whose purpose concerns the explicit computation of \eqref{rmk:degen_case} away from the singular locus).
\end{rmk}

Finally, our consideration in the present transversal context refers only to the $\sone$-invariant part $\Omega_0(G)$ of $\Omega(G)$ (cf. \eqref{eq:Dirac_composition_new}, Definitions \ref{def:Fourier_component_G} and \ref{def:index}). In fact, a formulation suitable to work for the $m$-th ($m\in\mathbb{Z}$) Fourier components $\Omega_m(G)$ exists by analogy. We omit the details here. See \cite[Corollary 1.13]{cheng2015heat} for a related formalism in the CR context. We content ourselves with the remark that for $m\in\mathbb{Z}$ the relevant local index density $I_m$ shall involve an additional term $p\,\delta_{p\mid m}\,e^{-m\frac{\mathrm d\omega_0}{2\pi}}$; namely
\[
I_m=\frac{1}{2\pi}p\,\delta_{p\mid m}\widehat{A}(\mathcal H)\wedge \ch\,\xi\wedge e^{-m\frac{\mathrm d\omega_0}{2\pi}}\wedge\omega_0
\]
(cf. \eqref{main} of Theorem \ref{thm:main_result}). Here $\delta_{p\mid m}=1$ if $p\mid m$ and $\delta_{p\mid m}=0$ if $p\nmid m$, and $p=\min_{u\in X}|H_u|$. See \cite[Corollary 1.13]{cheng2015heat} for analogy in this regard.

\bigskip

\appendix
\numberwithin{equation}{section}
\numberwithin{thm}{section}

\section{Construction of {$L$}-diffusion Measure on Orbifolds}
\label{appendix:L_diffusion}
The existence and the uniqueness (\cite[Theorems 1.3.4 and 1.3.6]{hsu2002stochastic}) of $L$-diffusion measures (cf.  \cite[Definition 1.3.1]{hsu2002stochastic}\footnote{$L$-diffusion can also be given as a solution to the martingale problem for $L$, cf. \cite[Proposition 3.2.1]{hsu2002stochastic}.}) are discussed through embedding the manifold into a Euclidean space.  This extrinsic approach cannot be directly applied to our orbifold case due to singularities. An alternative way is to construct the $L$-diffusion measure intrinsically. In fact ``... experience indicates that there are benefits to be gained from forcing oneself to work intrinsically. In particular, an intrinsic approach often reveals structure which is masked when one relies too heavily on extrinsic considerations ...'', said Stroock \cite[p. 165]{stroock2000introduction}.
Indeed \cite[Theorem 8.62]{stroock2000introduction} shows that for a connected, complete, separable $n$-dimensional Riemannian manifold $M$, if there exists an $\alpha\in\mathbb{N}$ such that for any fixed point $x_0\in M$, $\langle \mathrm{Ric}_x X_x,X_x\rangle\ge -{\alpha^2}/{n}(1+\mathrm{dist}(x,x_0)^2)\|X_x\|^2,\; \forall x\in M\text{ and } X_x\in  T_xM$, then the martingale problem for $\Laplace_{\mathscr{O}(M)}/2$ on $\mathscr{O}(M)$ (which is the  horizontal Laplacian $\Laplace_B/2$ in \cite{stroock2000introduction}) is well-posed; further the martingale problem for $\Laplace_M/2$ on the manifold $M$ is also well-posed. We can borrow this argument since the orthonormal frame bundle of our orbifold $M$ is smooth and compact.  For going down to the orbifold $M$, see e.g. \cite[Chapter 20.5, p. 589)]{del2017stochastic} for a special orbifold; nevertheless the arguments there basically apply here too (cf. Appendix \ref{appendix:adaptation}).

\section{Adaptation from Manifolds to Orbifolds}
\label{appendix:adaptation}

For the orbifold $M$ endowed with a Riemannian metric $g$ (see e.g. \cite[Proposition 2.20]{moerdijk2003introduction}
or \cite{wolak2015orbifolds}), the Laplacian $\Laplace:C^\infty(M)\to C^\infty(M)$ on $M$ is defined as follows (see \cite[Subsection 3.1]{gordon2012orbifolds}). Let $(\widetilde U,G,U,\pi)$ be any orbifold chart (with $\pi:\widetilde U\to\widetilde U/G\cong U\subset M$) with associated ($G$-invariant) Riemannian metric $g_{\widetilde U}$ and the associated Laplacian $\Laplace_{\widetilde U}$. For $f\in C^\infty(U)$ the function $\Laplace_{\widetilde U}(f\circ\pi)$ on $\widetilde U$ is $G_U$-invariant since $G$ acts isometrically on $\widetilde U$. Thus $\Laplace f$ is defined by the condition
\[
(\Laplace f)\circ \pi=\Laplace_{\widetilde U}(f\circ \pi).
\]
Other geometric operations (e.g. covariant derivatives in Theorem \ref{thm:grad_est}) can be defined on orbifolds in a similar fashion.

Secondly, given an orbifold chart $(\widetilde U,G,U,\pi)$ on $M=X/\sone$ it is true that
\begin{equation}
\label{eq:A1}
\widetilde U\times\sone/G \cong \pi^{-1}(U)\subset X,
\end{equation}
where $G$ acts diagonally and freely. One thinks of $X$ as an orbifold principal $\sone$-bundle on $M$ due to the $G$-action on $\widetilde U\times \sone$. As mentioned in \eqref{eq:Lplace^H_X}, this facilitates the definition of the transversal Laplacian $\Laplace_X^H$ (on $X$) in a way similar to that on the ordinary principal frame bundle.

To see \eqref{eq:A1}, fix $p\in X$ and a small slice $W=W_p\subset X$ transversal to the $\sone$-orbit of $p$. Here $W$ is the same as the smooth $G_y$-invariant manifold $V$ in the last paragraph of \cite[p. 173]{duistermaat2013heat}, where $G_y\subset\sone$ is the (finite) isotropy subgroup of $y(=p\text{ here})$. Then $W$ gives rise to an (smooth) orbifold chart $(\widetilde U,G,U,\pi)$ with $\widetilde U/G=U$ where $W$ is rewritten as $\widetilde U$ (\cite[p. 173]{duistermaat2013heat}). The mapping (where $\pi:X\to X/\sone$) $f:\widetilde U\times \sone\rightarrow \pi^{-1}(U)$ by $(w,s)\mapsto s^{-1}w$ induces the map $\overline{f}:\widetilde U\times\sone/G\to \pi^{-1}(U)$. Suppose that $f(w_1,s_1)=f(w_2,s_2)$ then $[w_1]=[w_2]$ in $M=X/\sone$. Back to the orbifold chart $\widetilde U(\subset X)$ on $M$, $w_1$ and $w_2$ shall be connected by an element $h\in G$ such that $s_1s_2^{-1}=h$. This yields $[(w_1,s_1)]=[(w_2,s_2)]$ in $\widetilde U\times\sone/G$ since $h\in G$, hence that $\overline{f}:\widetilde U\times\sone/G\to \pi^{-1}(U)$ is injective. Since $f$ is clearly surjective, $\overline f$ is thus an isomorphism, proving \eqref{eq:A1}.

\section{Transversally Spin Structure}
\label{appendix:spin}
In this work, the odd-dimensional manifold $X$ (of dimension $n+1$ with an $\sone$-invariant metric) is required to be orientable, and the horizontal part $\mathcal H\subset TX$ (i.e. the space of those tangents orthogonal to the $\sone$-orbits) inherits an orientation. For $n\ge 3$, a spin structure on $\mathcal H$ over $X$ is a principal $\mathrm{Spin}(n)$-bundle $\mathscr{SP}(\mathcal H)$ together with a two-sheeted covering $\widetilde\sigma:\mathscr{SP}(\mathcal H)\longrightarrow \mathscr{SO}(\mathcal H)$ such that for all $p\in \mathscr{SP}(\mathcal H)$ and all $g\in\mathrm{Spin}(n)$, $\widetilde\sigma(pg)=\widetilde\sigma(p)\widetilde\sigma_O(g)$ where $\widetilde\sigma_O:\mathrm{Spin}(n)\to\mathrm{SO}(n)$ is the universal covering homomorphism, cf. \cite[p. 80]{lawson2016spin}. The necessary and sufficient condition to have a spin structure on $\mathcal H$ is the vanishing of the second Stiefel-Whitney class $w_2(\mathcal H)$ of $\mathcal H$ (\cite[p. 79]{lawson2016spin}).

For a spin manifold $Z$, $\sone$-actions on $Z$ can be categorized into two types: odd type and even type. We say that an $\sone$-action on $Z$ is of even type if it lifts to an (compatible) action on $\mathscr{SP}(Z)$ and is of odd type otherwise (see \cite[p. 295]{lawson2016spin}). We extend this definition to the transversally spin manifold $X$ and require that the $\sone$-action on $X$ is of transversally even type (namely the $\sone$-action on X lifts to an action on $\mathscr{SP}(\mathcal H)$).

\section{lower bound and gradient estimate on logarithmic heat kernel}
\label{appendix:distant}
As indicated previously we want to show for distant points a similar lower bound estimate in Proposition \ref{prop:ker_est_M_lower} 
(upper bounds for them, see our Proposition \ref{prop:ker_est_M_upper}); it is connected to gradient estimate on the logarithmic heat kernel (cf.  Remark \ref{gradientlogheatkernel}).
Such a result on manifolds is treated in \cite{hsu2002stochastic}.   Here we want it 
for orbifolds, of which the proof below is much shorter and more in line with \cite{chavel1984eigenvalues} in place of some probabilistic arguments of \cite{hsu2002stochastic} (which might require extra work for adaptation to orbifolds).  


Retaining the notation of Proposition \ref{prop:ker_est_M_lower} assume that $y$ is not near to $x$. 
Suppose the special case where there exists a point $w\in M$ with a $\delta$-ball $B_{\delta}(w)$ such that 
$d(x, w)+d(w, y)=d(x, y)$ and for every $z\in B_{\delta}(w)$, $z$ is near to $y$ and to $x$ in the sense of Proposition \ref{prop:ker_est_M_lower}.  
By the semigroup property $p_M(t+s, x, y)=\int_M p(t, x, z)p_M(s, z, y)dz$ and the positivity of $p_M$ we have, 
via Proposition \ref{prop:ker_est_M_lower} for near points, 
\begin{equation}
\label{semigroup}
\begin{split}
p_M(t, x, y)&\ge \int_{B_{\delta}(w)}p_M(\frac{t}{2}, x, z)p_M(\frac{t}{2}, z, y)dz\\
&\ge C_1^2 \int_{B_{\delta}(w)}e^{-(\frac{d^2(x, z)+d^2(z, y)}{2\cdot t/2})}/(\pi t)^n dz\\
&= C_1^2 {e^{-\frac{d^2(x, y)}{2t}}} \int_{B_{\delta}(w)} e^{-\frac{f(z)}{2t}} dz/(\pi t)^n
\end{split}
\end{equation}
where $f(z) = 2d^2(x, z) + 2d^2(z, y) -d^2(x, y)$.  It is easily seen that $f(z)\ge 0$ with the minimum $f(w)=0$ by our choice of 
$w$.   Hence (by the smoothness of $f(z)$ due to the near-point assumption) 
$-f(z)\ge -C_2|z|^2$ for some constant $C_2>0$ in $B_{\delta}(w)$ (cf. Lemma \ref{:near_plemmats} and lines 
above for the orbifold distance).   Upon rescaling $z \to \sqrt{2t/C_2}z$ the preceding integral over $B_{\delta}$ is of order 
$(\delta\sqrt{2t/C_2})^n$ (where $n={\rm dim}\, M$).  This, together with \eqref{semigroup},  
yields the desired lower bound estimate for this special case.   Remark that the above $f$ is similar to the one appearing in 
the recent work \cite[Corollary 1]{barilari2012small} called {\sl hinged energy function} there.  

The strategy for a proof of the general case is not difficult.    Given $x, y \in M$, choose a geodesic $\gamma:[0, \ell]\to 
M$ connecting $x$ and $y$ with $\ell= d(x, y)$ (by an orbifold version of 
Hopf-Rinow Theorem, cf. \cite[p. 35]{bridson1999metric}).   Let $N>0$ and $\delta>0$ be such that by setting 
the points $w_i = \gamma(i\ell/N)$ ($i=0, 1, 2, \cdots, N$) with $w_0=x$ and $w_{N}=y$,  every point $z_i\in 
B_{\delta}(w_i)$ is near to every point $z_{i-1}\in B_{\delta}(w_{i-1})$ and $z_{i+1}\in B_{\delta}(w_{i+1})$
($i=1, 2, \cdots, N-1$).  The compactness of $M$ ensures that the choice of $N$, $\delta$ can be made independently of $x$ and $y$.  Now the same pattern of reasoning as the preceding special case for a similar integration,  
over intermediate variables $z_1, z_2, \cdots, z_{N-1}$ in the $\delta$-balls (via the semigroup 
property), of $e^{-f/2t}$ where $f=f(z_1, z_2, \cdots, z_{N})\ge 0$ is smooth with the minimal value $f(w_1, w_2, \cdots, w_{N-1})=0$ (i.e. $f=N(d^2(x, z_1)+d^2(z_1, z_2)+\cdots+d^2(z_{N-1}, y))-d^2(x, y)$), leads to the desired lower bound estimate for the general case of distant points.   The details are omitted.

\section{Some Results in Probability}
\label{subsection:Some Results in Probability}
For the geometrically minded reader, some basic results in probability, including Brownian motion, It\^o's formula, martingale and Feynman-Kac formula are briefly introduced here. For the sake of simplicity, we discuss them on $\mathbb{R}^n$; they can be generalized on manifolds. The main reference is \cite{klebaner2005introduction}.

\begin{defn}(cf. \cite[p. 56]{klebaner2005introduction})
\label{def:Brownian_motion}
Brownian motion $\{B(t)\}_{t\ge 0}$ in $\mathbb{R}$ is a stochastic process with the following properties:
\begin{enumerate}
\item[(i)](independence of increments) $B(t)-B(s)$ for $t>s$ is independent of the past,
\item[(ii)](normal increments) $B(t)-B(s)$ has normal distribution with mean $0$ and variance $|t-s|$ (denoted by $\mathcal N(0,|t-s|)$),
\item[(iii)] (continuity of paths) $B(t)$ are continuous functions of $t(\ge 0)$.
\end{enumerate}
\end{defn}

\begin{rmk}
Brownian motion $B(t)$ can be extended as a pair $\{ B(t),\mathcal F_t\}_{t\ge 0}$ where $\mathcal F_t$ is a filtration and $B(t)$ is an adapted process; see \cite[Remark 3.1]{klebaner2005introduction} for more. Brownian motion in high dimension, say $\mathbb{R}^n$, is defined as   $B(t)=\big(B^1(t),\dotsc ,B^n(t)\big)$ where $B^1,\cdots,B^n$ are independent, one-dimensional Brownian motions as just defined.
\end{rmk}

It follows that the conditional distribution of $B(t+s)$ is $\mathcal N(x,t)$ if $B(s)=x$. The transition function $p(y,t;x,s)$ is then given by
\[\begin{split}
p(y,t;x,s)&=\mathbb{P}\big(B(t+s)\le y|B(s)=x\big)\\
&=\mathbb{P}\big(B(t)\le y|B(0)=x\big)\text{ (also denoted as }\mathbb{P}_x(B(t)\le y))\\
&=\int_{-\infty}^y p_t(x,y_1)\mathrm dy_1\text{ where }p_t(x,y)=\frac{1}{\sqrt{2\pi t}}e^{-\frac{(y-x)^2}{2t}}.
\end{split}\]
This idea, together with the tower property $\mathbb{E}[X]=\mathbb{E}\left[\mathbb{E}[X|Y]\right]$, reappears in \eqref{eq:kernel computation_expectation2} for expressing the relation between heat kernels $p_X^D$ and $p_X$.

Unlike continuous deterministic functions, the quadratic variation of Brownian motion $B(t)$ is non-zero. More precisely,
\begin{defn}(Quadratic variation, cf. \cite[(3.8)]{klebaner2005introduction})
The quadratic variation of Brownian motion $B(t)$ is defined as $[B,B](t)=\lim\sum_{i=1}^n|B(t_i^n)-B(t_{i-1}^n)|^2$ where the limit is taken over all shrinking partition of $[0,t]$ with $\delta_n\coloneqq \max_{i}(t_{i}^n-t_{i-1}^n)\to 0$ as $n\to\infty$.

\begin{thm}(cf. \cite[Theorem 3.4]{klebaner2005introduction})
We have $[B,B](t)=t$.
\end{thm}
\end{defn}


Next, we turn to martingale and the zero mean property of It\^o integral.

\begin{defn}(Martingale, cf. \cite[Definition 3.6, Remark 3.3]{klebaner2005introduction})
A stochastic process $\{X(t\}_{t\ge 0})$ is a martingale if for any $t$, $\mathbb{E}|X(t)|<\infty$ and for any $s>0$, $\mathbb{E}\big[X(t+s)|\mathcal F_s\big]=X(t)\,\text{ a.s.}$ where $\mathcal F_t=\sigma(X(u):0\le u\le t)$ the $\sigma$-field generated by the values of the process up to time $t$.
\end{defn}
\begin{lemma} (cf. \cite[Theorem 3.7]{klebaner2005introduction})
Brownian motion $B(t)$ is a martingale.
\end{lemma}

In what follows we are going to introduce the classical Feynman-Kac formula, of which the one that Hsu has used in \cite{hsu2002stochastic} is an adaptation.

\begin{thm}(Zero mean property, cf. \cite[Theorem 4.3]{klebaner2005introduction})
Let $X(t)$ be a regular adapted process such that $\int_0^T X^2(t)\mathrm dt<\infty$ with probability one. If $\int_0^T\mathbb{E}[X^2(t)]\mathrm dt<\infty$ holds for It\^o integral $\int_0^T X(t)\mathrm dB(t)$, then $\mathbb{E}\left[\int_0^T X(t)\mathrm dB(t)\right]=0$.
\end{thm}

\begin{thm}(It\^o's formula, cf. \cite[Theorem 4.13]{klebaner2005introduction})
\label{thm:Ito_formula}
If $B(t)$ is a Brownian motion on $[0,T]$ and $f(x)\in C^2(\mathbb{R})$, then $f\big(B(t)\big)=f\big(B(0)\big)+\int_0^tf'\big(B(s)\big)\mathrm dB(s)+\frac{1}{2}\int_0^tf''\big(B(s)\big)\mathrm ds\text{ for any }t<T$ (or $\mathrm df\big(B(t)\big)=f'\big(B(s)\big)\mathrm dB(s)+\frac{1}{2}f''\big(B(s)\big)\mathrm ds$ for short).
\end{thm}

\begin{rmk}
It\^o's formula extends to higher dimensions (see e.g. \cite[p. 116]{klebaner2005introduction}).
\end{rmk}

It\^o's formula together with the zero mean property enters into the essence of Feynman-Kac formula which in turn builds a connection between partial differential equation and probability.

\begin{thm}(Feynman-Kac formula, cf. \cite[Theorem 6.8]{klebaner2005introduction})
\label{thm:F-K_formula}
For given bounded functions $r(x,t)$ and $g(t)$ the solution, if any, to ${\partial f(x,t)}/{\partial t}+L_t f(x,t)=r(x,t)f(x,t),\; f(x,T)=g(x)$ is uniquely given by $C(x,t)=\mathbb{E}[e^{-\int_t^T r\big(X(u),u\big)\mathrm du}g\big(X(T)\big)|X(t)=x]$. Here $L_t$ denotes the generator of the stochastic process $X(t)$.
\end{thm}

\begin{rmk}
\label{rmk: FK}
For Brownian motion in $\mathbb{R}^n$ its generator $L_t=1/2 \Laplace$ and thus the solution to equation ${\partial f}(t,x)/{\partial t}+L_t f(x,t)=0,\; f(x,T)=g(x)$ is given by $\mathbb{E}[g\big(B(T)\big)|B(t)=x]$ (cf. \cite[Theorem 6.6]{klebaner2005introduction}). Equivalently, the solution to the standard heat equation ${\partial f}(t,x)/{\partial t}={1}/{2}\Laplace f(x,t),\; f(x,0)=g(x)$ is given by $\mathbb{E}_x [g\big(B(t)\big)]\coloneqq \mathbb{E}\left[ g\big(B(t)\big)|B(0)=x\right]$. 
\end{rmk}

The appropriate version of the Feynman-Kac formula used in geometry (cf. \cite[Theorem 7.2.1]{hsu2002stochastic} or \cite{bismut1984atiyah}) is an adaptation of the classical one above.  It will be further extended to the locally free case (see \eqref{eq:heat_eq_O(M)} in Subsection \ref{subsection:Prob} and Appendix \ref{appendix:adaptation}).
\end{document}